%% file: main.tex
\documentclass{siamart171218}

\input{ex_shared}

\ifpdf
\hypersetup{
  pdftitle={An Example Article},
  pdfauthor={D. Doe, P. T. Frank, and J. E. Smith}
}
\fi

\begin{document}

\maketitle

\begin{abstract}
High-fidelity, high-resolution numerical simulations are crucial for studying complex multiscale phenomena in fluid dynamics, such as turbulent flows and ocean waves. However, direct numerical simulations with high-resolution solvers are computationally prohibitive. As an alternative, super-resolution techniques enable the enhancement of low-fidelity, low-resolution simulations. However,  traditional super-resolution approaches rely on paired low-fidelity, low-resolution and high-fidelity, high-resolution datasets for training, which are often impossible to acquire in complex flow systems. To address this challenge, we propose a novel two-step approach that eliminates the need for paired datasets. First, we perform unpaired domain translation at the low-resolution level using an Enhanced Denoising Diffusion Implicit Bridge. This process transforms low-fidelity, low-resolution inputs into high-fidelity, low-resolution outputs, and we provide a theoretical analysis to highlight the advantages of this enhanced diffusion-based approach. Second, we employ the cascaded Super-Resolution via Repeated Refinement model to upscale the high-fidelity, low-resolution prediction to the high-resolution result. We demonstrate the effectiveness of our approach across three fluid dynamics problems. Moreover, by incorporating a neural operator to learn system dynamics, our method can be extended to improve evolutionary simulations of low-fidelity, low-resolution data. 
\end{abstract}

\begin{keywords}
super-resolution, downscaling, diffusion models, fluid dynamics, unpaired domain translation
\end{keywords}

\begin{AMS}
65C60, 65M22, 65M50, 68T07, 76F55
\end{AMS}

\section{Introduction}
Simulating high-fidelity and high-resolution (HFHR) data is of great importance in many scientific problems. However, obtaining HFHR data via direct numerical simulations (DNS) is computationally expensive and thus often impractical for large-scale or long-term simulations. Super-resolution (SR) has emerged as a computationally efficient, data-driven surrogate strategy to generate HFHR outputs from low-fidelity, low-resolution (LFLR) simulations. While deep learning models, notably convolutional neural networks (CNNs) \cite{o2015introduction,sun2020downscaling} and generative adversarial networks (GANs) \cite{goodfellow2020generative,stengel2020adversarial}, have demonstrated success in image-based SR, their adaptation to scientific simulations faces two critical limitations. First, many existing SR models rely on paired LFLR-HFHR data for supervised training. Yet in practice, acquiring paired data is often impossible. For example, in chaotic systems, trajectories starting from two  close initial conditions can diverge rapidly. Consequently, simulation data from a high resolution (HR) solver and the corresponding data from a low resolution (LR) solver, both starting from the same initial condition at different resolutions, may deviate significantly over time, making it impossible to construct a paired dataset. Second, even if paired data were available, LFLR simulations cannot be treated as low-resolution approximations of HFHR data. This is because the LR solvers often fail to capture small-scale dynamics, which can accumulate and eventually influence the large-scale dynamics. Consequently, LFLR simulations exhibit systematic biases distinct from high-fidelity, low-resolution (HFLR) data, which refers to the downsampled version of HFHR data.

These challenges motivate the first central question of this paper:

\textit{Given two unpaired datasets of LFLR simulation data (from an LR solver) and HFHR simulation data (from an HR solver), how can we effectively enhance the fidelity of the LFLR simulation data?}

In addition to the time-independent scenarios, many scientific applications, such as turbulent fluid flows and ocean wave modeling, require enhancing evolutionary simulations where temporal consistency is critical. In these systems, errors in LFLR simulations corrupt instantaneous results and propagate over time, destabilizing long-term trajectories. More importantly, as discussed earlier, in a chaotic system, it is impossible for the LR solver to produce meaningful solution trajectories because a small variation in the initial condition will lead to a completely different solution at later times. This motivates our second central question:

\textit{Given two unpaired trajectory datasets of LFLR simulations (from an LR solver) and HFHR simulations (from an HR solver), how can we enhance the fidelity of the LFLR trajectory simulations?}

In \cref{sec:problem_setting}, we mathematically formulate the two central questions and introduce the necessary notations. We then provide a high-level overview of our approach and conclude with a summary of the main contributions of this work. The detailed methodology is presented in \cref{sec:method}. In \cref{sec:related}, we provide a detailed literature review and highlight the advantages of our approach compared to existing methods.

\subsection{Problem description and methodology overview}\label{sec:problem_setting}

The first central question this paper proposes to study is to construct a data-driven enhancement approach to translate LFLR \textbf{snapshot data} into HFHR data using two unpaired datasets. Formally, we let $u^h$ denote HFHR data generated by an HR solver, and $u^l$ denote LFLR data generated by an LR solver. Consequently, this question can be formulated as an unpaired domain translation problem between two empirical distributions $\{ u^l \}$ and $\{ u^h \}$. Since $u^l$ and $u^h$ have different resolutions, we introduce a restriction operator $\mathcal{R}$ that downsamples the HFHR data to obtain its high-fidelity, low-resolution (HFLR) counterpart, defined as $\tilde{u}^h:=\mathcal{R} u^h$. This allows us to generate a paired dataset $\{ \tilde{u}^h, u^h \}$, where $\tilde{u}^h$ serves as an intermediate representation that bridges the domain gap. Specifically, rather than directly learning a mapping from $u^l$ to $u^h$, we decompose the problem into two sub-tasks:
\begin{itemize}
    \item \textbf{Debiasing}: Learn a transformation $\mathcal{T}$ that maps the LFLR $u^l$ into its HFLR counterpart $\tilde{u}^h$ using unpaired datasets $\{ u^l \}$ and $\{ \tilde{u}^h \}$. This step bridges the fidelity gap between the biased LFLR data and the target HFLR data.
    \item \textbf{Super-Resolution (SR)}: Learn an SR model $\mS$ to reconstruct $u^h$ from $\tilde{u}^h$ using paired dataset $\{ \tilde{u}^h, u^h \}$. This step bridges the resolution gap between the HFLR data and the HFHR data.
\end{itemize}
The diagram is presented in \cref{fig:diag}. As illustrated in the figure, recovering the HFHR data involves two critical considerations during the debiasing step. First, the large-scale structure of the data must be preserved. Second, the translation process should generate the desired fine-scale details at the low-resolution level. To address these challenges, we propose an enhanced version of Diffusion Domain Interpolation Bridge (DDIB) \cite{su2022dual} that robustly transforms LFLR data $u^l$ into HFLR data $\tilde{u}^h$, simultaneously preserving the large-scale structure and generating the desired fine-scale details. For the SR step, which is inherently a pseudo-inverse problem, we treat it as a conditional sampling problem and implement a cascaded Super-Resolution via Repeated Refinement (SR3) \cite{saharia2022image}. Detailed descriptions of our methods are provided in \cref{sec:method}.

\begin{figure}[ht]
    \centering
    \includegraphics[width=0.8\textwidth]{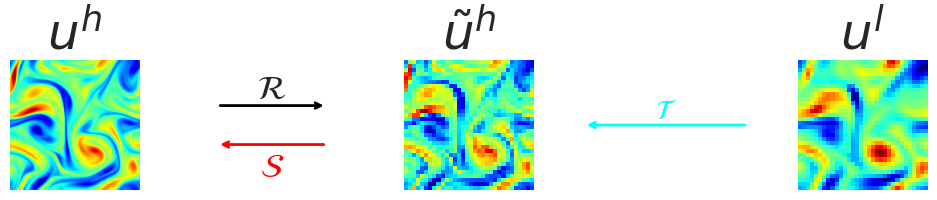}
    \caption{Relationship among HFHR (left), HFLR (middle), and LFLR (right) data in time-snapshot super-resolution.}
    \label{fig:diag}
\end{figure}

The second central question of this paper is how to convert LFLR \textbf{trajectory data} into HFHR trajectory data using unpaired datasets.  We denote the trajectory of $n$ temporal snapshots with a fixed time step $\Delta t$ as $\boldsymbol{u} := \{u^1, \dots, u^n\}$. Here, $u^\tau$ (where $\tau = 1, \dots, n$) denotes the data at time $t = \tau \Delta t$. Depending on the context, $u^\tau$ may represent HFHR data $u^{\tau,h}$, LFLR data $u^{\tau,l}$, or HFLR data $\tilde{u}^{\tau, h}$. For example, $\boldsymbol{u}^h := \{u^{1,h}, \dots, u^{n,h}\}$ represents an HFHR trajectory simulation. Thus the problem can be formulated as an unpaired domain translation task between the empirical datasets $\{\boldsymbol{u}^l\}$ and $\{\boldsymbol{u}^h\}$. 

A naive approach to enhance the LFLR trajectory simulation is to refine it snapshot by snapshot using the method designed for the time-snapshot setting, which is referred to as $\textit{snapshot-wise refinement}$. However, because the LR solver introduces inaccuracies that accumulate over time, this snapshot-wise refinement becomes less reliable at later times. 

To address this limitation, we propose two complementary approaches that learn the dynamic evolution of trajectories from high-fidelity data at different resolution levels. The first approach trains $\tilde{\mathcal{G}}^h$ to capture the system’s evolution at the low resolution level using HFLR $\{ \tilde{\bu}^h \}$, which is referred to as $\textit{HFLR dynamics learning}$. The second approach trains $\mathcal{G}^h$ to learn the evolution at the high resolution level using HFHR $\{ \bu^h \}$, which is referred to as $\textit{HFHR dynamics learning}$. \cref{fig:diag_evo} illustrates and compares these three strategies: the snapshot-wise enhancement, HFLR dynamics learning, and HFHR dynamics learning. Note that the translation $\mathcal{T}$ and the super-resolution $\mathcal{S}$ remain the same as those employed in the time-snapshot setting. Although the figure illustrates the generation of the HFHR prediction at $t=t_n$, the same procedure can be applied to any or all time steps in the trajectory.

Among these three strategies, snapshot-wise refinement suffers from increasing inaccuracy over time, while HFHR dynamics learning requires a model capable of resolving small-scale dynamics at the high resolution level. Such a model must be sufficiently large and complex; however, as discussed in \cite{molinaro2024generative,mardani2023generative,rampal2024enhancing}, even large deterministic frameworks struggle to accurately capture these small-scale dynamics in practice. Consequently, we adopt the HFLR dynamics learning in this paper. This approach follows the same general structure as the snapshot-wise setting, but it additionally incorporates a neural operator to model the system dynamics over time. The complete methodology will be detailed in \cref{sec:method}. For completeness, we also include a numerical comparison of all three strategies in \cref{sec:numerical}. Source code is available at \url{https://github.com/woodssss/Unpaired_SR_demo_code}.

\begin{figure}[ht]
    \centering
    \includegraphics[width=0.32\textwidth]{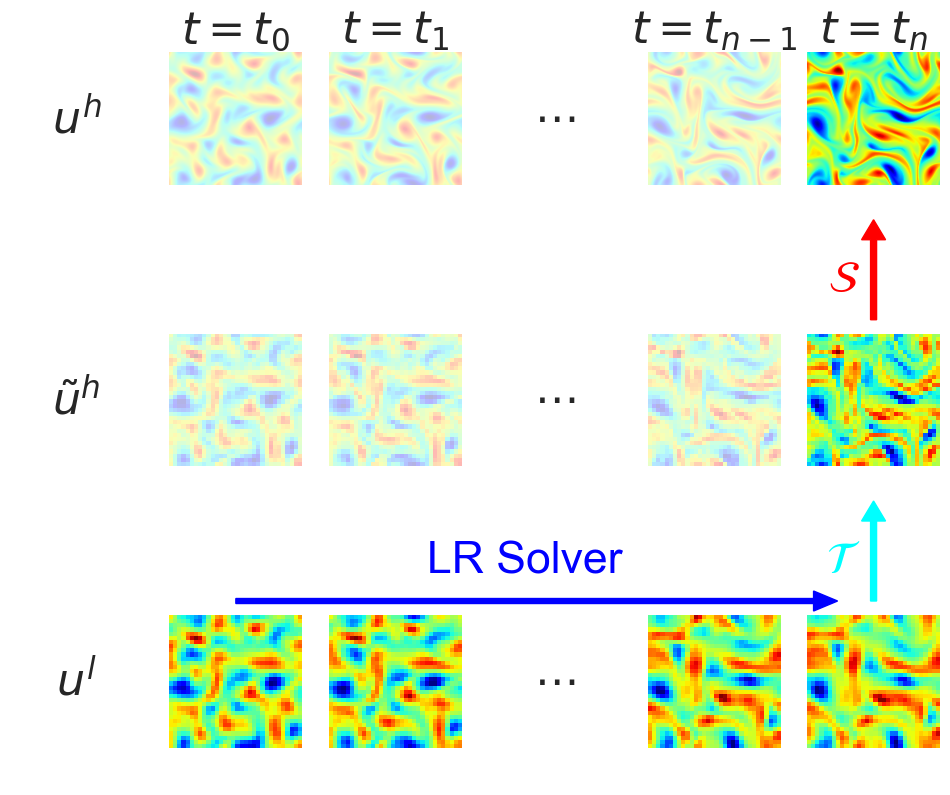}
    \includegraphics[width=0.32\textwidth]{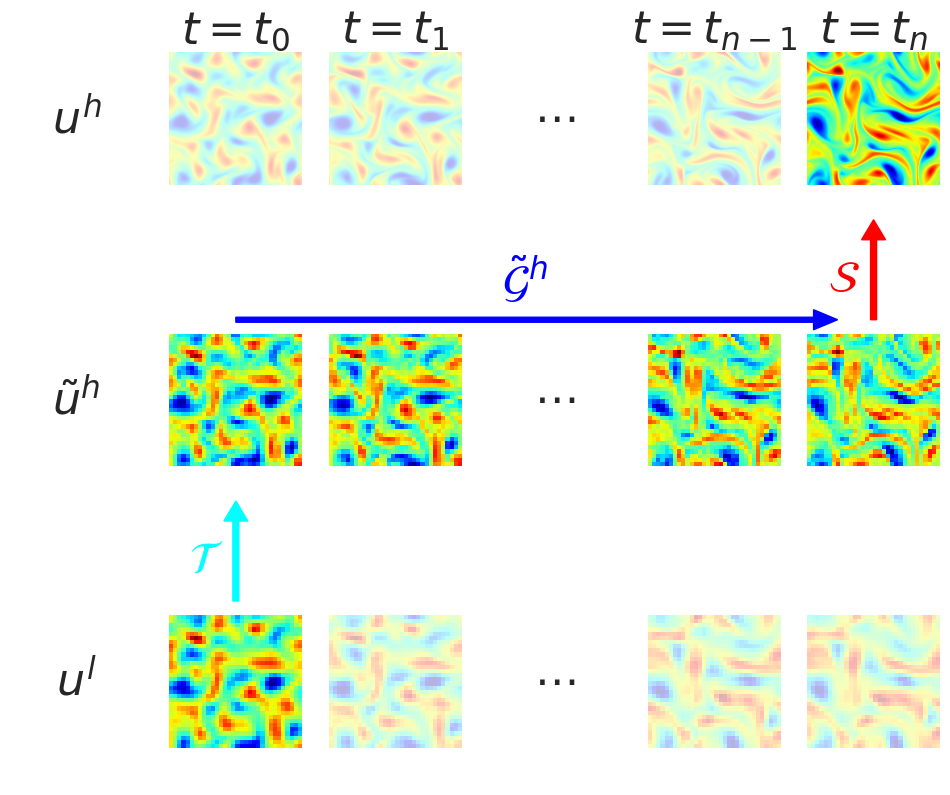}
    \includegraphics[width=0.32\textwidth]{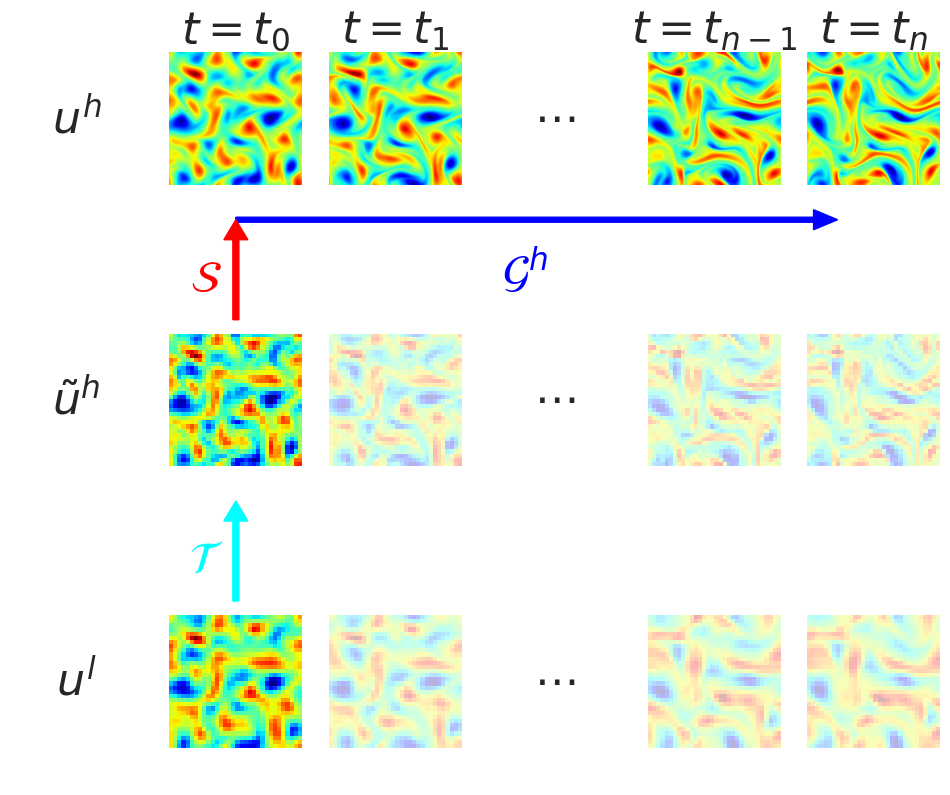}
    \caption{Comparison of snapshot-wise refinement (left), HFLR dynamics learning (middle), and HFHR dynamics learning (right) for enhancing the trajectory data.}
    \label{fig:diag_evo}
\end{figure}

\textbf{Main Contributions.} The main contributions of this paper can be summarized as follows:
\begin{itemize}
    \item We introduce an enhanced DDIB (EDDIB) approach that significantly improves the performance of unpaired domain translation, particularly under limited training data constraints. We provide theoretical analyses and numerical experiments to demonstrate its effectiveness.

    \item We propose a two-step diffusion model-based approach for unpaired SR in fluid dynamics. This approach preserves large-scale coherent structures while recovering statistically consistent fine-scale details, addressing the limitations of existing SR approaches.

    \item By integrating our diffusion-based approach with neural operator, we enable accurate and stable long-term enhancement of LFLR trajectories in fluid dynamics simulations. This hybrid approach ensures temporal stability and mitigates error accumulation, achieving high-fidelity predictions over extended time horizons.

\end{itemize}

\subsection{Related works}\label{sec:related}
Downscaling has been widely used in climate and weather modeling to enhance coarse outputs produced by global or LR models. It generally falls into two categories. In \textit{dynamical downscaling}, a global climate model (GCM) is first run to produce an LR output. This output then serves as the initial and boundary conditions for a regional climate model (RCM) or limited-area model \cite{adachi2020methodology,giorgi2012regcm4, hawkins2009potential}, which solves the governing equations at a finer resolution. This approach enables the capture of region-specific features, such as complex terrain and local circulations, that are not well represented in the coarse-scale model, but computationally expensive when running the RCM at HR. 

In contrast, \textit{statistical downscaling} leverages data-driven methods to learn mappings between LFLR and HFHR data. Recent breakthroughs in deep learning, especially in computer vision, have spurred the development of SR techniques that upsample LFLR inputs into HFHR outputs using advanced architectures. These SR approaches can be broadly classified into \textit{deterministic} and \textit{probabilistic} approaches. Deterministic approaches, including CNN-based methods \cite{dong2015image,reddy2023precipitation,sun2020downscaling} and RNN-based models \cite{zhang2021future}, are effective at capturing the ``mean-state'', because they rely on regression toward the mean during training; however, they tend to overlook data variance and thus struggle to resolve fine-scale details \cite{molinaro2024generative,mardani2023generative,rampal2024enhancing}. 

More recently, generative model-based SR methods have emerged as probabilistic approaches for capturing small-scale details in complex data. For example, GAN-based methods \cite{stengel2020adversarial,liu2022spatial} have proven effective in producing high-quality outputs, although they often suffer from training instability. Similarly, normalizing flow (NF)-based methods \cite{groenke2020climalign} model intricate data distributions effectively but tend to be computationally demanding, as they require deeper and more elaborate architectures to capture fine details in high-dimensional datasets. In contrast, diffusion-based models \cite{lu2024generative,  molinaro2024generative} have emerged as a promising alternative, offering stable training and the ability to capture fine-scale details in complex datasets. Despite their effectiveness, it is important to note that all these statistical downscaling and SR methods require paired LFLR and HFHR datasets for training, limiting their application when such paired data are unavailable.

\textbf{Domain alignment.} The task of translating $u^l$ into $\tilde{u}^h$ using unpaired datasets can be treated as an instance of unpaired domain alignment task. This involves learning the translation between two distributions without paired correspondences. Common approaches for such tasks include optimal transport (OT)-based methods, such as the Sinkhorn algorithm \cite{cuturi2013sinkhorn,peyre2019computational} or neural network approximations of transport maps \cite{korotin2022neural}, as well as generative models like GANs \cite{zhu2017unpaired,mao2017aligngan} and normalizing flows \cite{grover2020alignflow,sagawa2025gradual}. While effective, these methods inherently depend on predefined pairs of source and target domains, limiting their scalability. Specifically, ``paired domains'' in this context refer to task-specific source-target pairs, distinct from the paired LFLR-HFHR data with one-to-one correspondences. Scaling such frameworks to multiple domains would require a quadratic number of models relative to the number of domains, making them impractical for multi-domain applications.

A notable alternative is Stochastic Differential Editing (SDEdit) \cite{meng2021sdedit}, which circumvents task specificity by training a single diffusion model on the target domain. With this approach, samples from other domains are edited by injecting noise into the input and then guiding its denoising process using stochastic differential equations (SDEs) and the pre-trained diffusion model, enabling flexible domain translation while preserving the original structure. However, SDEdit faces a robustness challenge due to an inherent trade-off between fidelity and realism. In our context, this trade-off involves the need to maintain large-scale structural features while accurately recovering the desired statistical properties. Although recent efforts have enhanced the robustness of SDEdit \cite{li2024adbm, nie2022diffusion}, these methods remain primarily effective at removing extraneous biases such as spurious numerical errors or noise. In contrast, our problem requires generating physically consistent fine-scale details that are unresolved by the LR solver. A direct application of the SDEdit fails to address this challenge, as we further demonstrate later in \cref{sec:numerical}.

Another solution is the Diffusion Domain Interpolation Bridge (DDIB) \cite{su2022dual}, which employs two independently trained diffusion models, one for each domain, and uses the Probability Flow (PF) ODEs to map each domain into a shared latent space, thereby bridging the two domains through this shared latent space. Because each domain is learned separately and projected onto a common latent space, DDIB enables the reuse of these models to translate between any pair of domains. Moreover, the authors of \cite{su2022dual} show that DDIB can be viewed as a special case of optimal transport with regularization, enabling it to capture underlying correspondences between distributions. We therefore adopt DDIB to translate $u^l$ to $\tilde{u}^h$ at the low-resolution level. 

Despite its advantages, DDIB’s ability to translate between source and target domains depends on the quality of the two mappings; that is,  how effectively each domain is mapped into the shared latent space. In practice, ensuring high-quality mappings typically requires a sufficiently large dataset. While there is no analytical result on the minimum sample size needed, even relatively small datasets like CIFAR-10 \cite{krizhevsky2009learning} or MNIST \cite{lecun1998gradient} each contain 60,000 samples. Generating a comparable volume of high-fidelity, high-resolution (HFHR) simulations in scientific computing problems can be prohibitively expensive. To address this limitation, we propose an EDDIB method that requires relatively small datasets for training diffusion models.

\textbf{Most relevant work.} The most relevant work addressing similar problems, albeit using a different approach, is \cite{wan2023debias}. In this work, the goal is to transform an empirical LFLR sample distribution into samples from the corresponding HFHR distribution. Their method involves two steps: first, using an OT map to transform the LFLR data into HFLR representations, and then applying a conditional diffusion model to upscale the transformed data to high resolution. While this approach appears similar to ours, there are two key distinctions. 

First, their method relies on an OT map to translate LFLR data into HFLR form. While this OT map effectively recovers statistical properties, it fails to preserve large-scale structures within the LFLR data. For highly chaotic systems, where no true HFHR counterpart exists for a given LFLR state, their method is suitable. However, for systems where LFLR-HFHR pairs do exist and the LFLR data is roughly the lower-resolution version of the HFHR data (albeit with some bias), it is preferable to map the LFLR data to a sample within the HFHR distribution that preserves both large-scale structures and introduces the necessary fine-scale details. We have also explored an alternative OT method, neural OT \cite{korotin2022neural}, which preserves large-scale structures better but struggles to recover the desired small-scale details. In contrast, our method achieves both objectives simultaneously. Extensive numerical demonstrations are provided in \cref{sec:numerical}. Second, rather than employing a single conditional diffusion model to upscale the HFLR data, we adopt a cascaded SR3 \cite{saharia2022image} approach that iteratively refines the data using a sequence of diffusion models. This cascaded framework not only improves computational efficiency through parallel training and task decomposition but also yields superior high-resolution reconstructions by progressively refining details across scales.

\section{Methodology}\label{sec:method}
In this section, we present our numerical method. \Cref{sec:background} provides a brief introduction to the background of diffusion models. \Cref{sec:ddib background} details the vanilla DDIB method and the EDDIB method for the transformation $\mathcal{T}: u^l \rightarrow \tilde{u}^h$ is presented in \cref{sec:eddib}. In \cref{sec:sr}, we describe the application of cascaded SR3 models for the super-resolution step $\mathcal{S}: \tilde{u}^h \rightarrow u^h$, which refines these approximations to achieve high-fidelity outputs. \Cref{sec:fno} introduces the Fourier Neural Operator (FNO) \cite{li2020fourier}, utilized to learn the dynamic operators $\tilde{\mathcal{G}}^h$ for LR simulations and $\mathcal{G}^h$ for HR simulations.

\subsection{Background on diffusion models}\label{sec:background}
The score-based diffusion model \cite{song2020score, yang2023diffusion} aims to approximate a target data distribution $p_{data}(\bx)$ given a dataset $\{\bx_i \}_{i=1}^N$ in the unconditional setting. This method can also be extended to the conditional setting, where the goal is to approximate a conditional distribution  $p_{data}(\bx|\by)$ using a paired dataset $\{\bx_i, \by_i \}_{i=1}^N$. For simplicity, we focus on the unconditional case here. The score-based diffusion model operates in two stages: a forward diffusion process and a reverse generative process. 

\textbf{Forward process.} In the forward process, data is progressively corrupted with noise through a stochastic differential equation (SDE), transforming the original data distribution $p_{data}$ into a standard Gaussian distribution:
\begin{equation}\label{eqn:diff}
    \mathrm{d}\bx = f(\bx, t) \rd t + g(t)  \mathrm{d}\mathbf{w},    
\end{equation}
where $f(\bx, t)$ is the drift term, $\mathrm{d}\mathbf{w}$ is a standard Wiener process, and $g(t)$ is the diffusion coefficient. The initial condition is $\bx(0):=\bx \sim P_{data}(\bx)$. There are typically two types of SDEs used for the score-based diffusion models: Variance Exploding (VE) SDE and Variance Preserving (VP) SDE. In this paper, we adopt the VP SDE, as it transforms the original data distribution into an isotropic Gaussian distribution, which serves as the shared latent space in the DDIB. Specifically, we adopt the DDPM setting, where the drift term is defined as $f(\bx, t) = -\frac{1}{2}\beta(t)\bx$ and the diffusion coefficient is set to $g(t) = \sqrt{\beta(t)}$. This corresponds to a special case of the VP SDE:
\begin{equation}\label{eqn:forward}
\rd \bx = -\frac{1}{2} \beta(t) \bx \rd t + \sqrt{\beta(t)} \rd \mathbf{w}.
\end{equation}
The perturbed solution of this SDE at time $t$ with an initial condition $\bx(0)$ is:
\[
\bx(t) = \alpha(t)\bx(0) + \sigma(t) \beps, ~ \beps \sim \mN(0, \mI).
\]
In the above, $\beta(t)$ is a user-specific monotonically increasing function for $t \in [0, 1]$, $\alpha(t) = e^{-\frac{1}{2}\int_0^t \beta(s) \rd s}$, and $\sigma^2(t) = 1- \alpha^2(t)$. Typically, the marginal distribution at time $t=1$ approaches the standard Gaussian distribution 
\[
p(\bx(1)|\bx(0)) = \mN(\bx(1); \alpha(1) \bx(0) , (1-\alpha^2(1)) \mI),
\]
if $\beta(t)$ is selected such that $\lim_{t \rightarrow 1} \alpha(t) = 0$. 

\textbf{Reverse process.} The reverse process is described by the corresponding reverse time SDE that progressively transforms the standard Gaussian distribution back into the original data distribution.
\begin{equation}\label{eqn:reverse}
    \mathrm{d}\mathbf{x} = \left[ -\frac{1}{2} \beta(t) \bx -   \beta(t) \nabla_{\mathbf{x}} \log p_t(\bx)\right] \mathrm{d}t + \sqrt{\beta(t)}  \mathrm{d}\bar{\mathbf{w}},
\end{equation}
where $\mathrm{d}\mathbf{\bar{w}}$ represents a reverse-time Wiener process and $\log p_t(\mathbf{x})$ is the score function of the marginal distribution of the forward process at time $t$. Moreover, Song et al. \cite{song2020score} proved the existence of the probability flow ODE, which shares the same marginal distributions as the reverse-time SDE \eqref{eqn:reverse}:
\begin{equation}\label{eqn:pfode_ori}
    \frac{\rd \bx}{\rd t} = -\frac{1}{2} \beta(t) \bx - \frac{1}{2} \beta(t) \nabla_{\mathbf{x}} \log p_t(\bx).
\end{equation}

\textbf{Training and sampling.} In practice, the score function $\nabla_{\bx} \log p_t(\bx)$ is estimated using a score-matching objective \cite{song2020score}:
\begin{equation}\label{eqn:loss}
L(\theta):= \mathbb{E}_{t \sim U(0, 1), \bx \sim p(\bx), \beps \sim \mN(0, \mI)} [\| \sigma(t) S_{\theta}(\bx(t), t) + \beps \|_2^2]
\end{equation}
where $S_{\theta}(\bx(t), t)$ is a time-dependent neural network approximating the score function $\nabla_{\bx_t} \log p_{t}(\bx)$. 

Once the $S_{\theta}(\bx_t, t)$ is well trained, new samples can be generated by either solving the PF ODE \eqref{eqn:pfode_notation} or the reverse-time SDE \eqref{eqn:reverse}.

For sampling using the PF ODE, we denote the solution of the ODE driven by velocity field $v(\bx(t), t)$ from $t_1$ to $t_2$ by:
\begin{equation}\label{eqn:pfode_notation}
    \text{ODEsolve}(\bx(t_1), t_1, t_2; v) = \bx(t_1) + \int_{t_1}^{t_2} v(\bx(t), t) \rd t
\end{equation}
Using this formulation, a sample can be generated by 
$\bx(0) = \text{ODEsolve}(\bx(1), 1, 0; v_{\theta})$,
where $\bx(1) = \beps \sim \mN(0, \mI)$ and the velocity field is defined as:
\[
v_{\theta}(\bx(t), t) = -\frac{1}{2} \beta(t) \bx(t) - \frac{1}{2} \beta(t) S_{\theta}(\bx(t), t).
\]
In practice, any black-box ODE solver can be employed to perform this integration, making PF ODE-based sampling computationally efficient. On the other hand, reverse-time SDE sampling can be performed using any general-purpose SDE solver to integrate the reverse-time SDE \eqref{eqn:reverse}, and it typically produces higher-quality samples compared to PF ODE-based sampling. To further improve the quality of samples generated by the reverse-time SDE, Predictor-Corrector (PC) samplers can be employed. These samplers combine numerical SDE solvers with score-based Markov Chain Monte Carlo (MCMC) approaches \cite{song2020denoising}, offering enhanced performance for high-fidelity sampling, thus we adopt them for the SR step.

\subsection{Dual diffusion implicit bridges}\label{sec:ddib background}
The dual diffusion implicit bridges (DDIB) method \cite{su2022dual} addresses unpaired domain translation by independently training separate diffusion models for the source and target domains. Each diffusion model used for mapping its respective domain to a shared latent space, enabling the translation of a sample from the source domain to a corresponding sample in the target domain through a two-step process. First, the sample from the source domain is encoded into a latent representation within the shared latent space, and then it is decoded into the corresponding sample in the target domain. When the VP SDE is employed in the forward diffusion process, the shared latent space typically corresponds to a standard Gaussian distribution.

Note that at the low resolution level, we aim to transform the LFLR data $u^l$ to its HFLR counterpart $\tu$ using two unpaired datasets $\{ u^l_i\}_{i=1}^N$ and $\{ \tilde{u}^h_j\}_{j=1}^M$. To achieve this, we adapt the DDIB to our problem by training two separate unconditional diffusion models: $S^l_{\xi}$ for dataset $\{ u^l_i\}_{i=1}^N$ and $\tilde{S}^h_{\zeta}$ for dataset $\{ \tilde{u}^h_j\}_{j=1}^M$. The training procedure for $S^l_{\xi}$ is outlined in \cref{alg:utilde} and the same procedure can be extended to train $\tilde{S}^h_{\zeta}$.
\begin{algorithm}[ht]
\caption{Unconditional Diffusion Model}
\label{alg:utilde}
\begin{algorithmic}[1]
\REQUIRE Training datasets $\mathcal{A} = \{ u^l_i\}_{i=1}^N $, noise scheduling function $\alpha(t), \sigma(t)$, batch size $B$ and max iteration $Iter$
\STATE Initialize $k=0$
\WHILE{$k<Iter$}
    \STATE Sample $\{ u^l_j \}_{j=1}^B \sim \mathcal{A}$ 
    \STATE $t \sim U[0, 1]$
    \STATE $\beps_j \sim \mN(0, \mI)$ for $j=1, \cdots, B$
    \STATE Compute $u_j(t) = \alpha(t) u_j + \sigma(t) \beps_j$
    \STATE Update $\xi$ using the Adam optimization algorithm \cite{kingma2014adam} to minimize the empirical loss:
    \[
    L(\xi)= \frac{1}{B} \sum_{j=1}^B \|  \beps_j + \sigma(t) S^l_{\xi}(u_j(t), t) \|_2^2
    \]
    \STATE $k \gets k+1$
\ENDWHILE
\RETURN Diffusion model $S^l_{\xi}(u(t), t)$
\end{algorithmic}
\end{algorithm}

Once these two diffusion models $S_\xi^l$ and $\tilde{S}_\zeta^h$ are well trained, the translation is achieved by the following two steps:
\begin{itemize}
    \item Latent encoding: $z = \text{ODEsolve}(u^l, 0, 1; \mTl_{\xi})$,
    \item Decoding: $\tilde{u}^h = \text{ODEsolve}(z, 1, 0; \mTt_{\zeta})$,
\end{itemize}
where the velocity fields are 
\begin{equation}\label{eqn:velo1}
\mTl_\xi(u^l(t), t) = -\frac{1}{2} \beta(t) u^l(t) - \frac{1}{2}\beta(t)S^l_{\xi}(u^l(t), t), 
\end{equation}
and 
\begin{equation}\label{eqn:velo2}
\mTt_\zeta (\tu(t), t) = -\frac{1}{2} \beta(t) \tu(t) - \frac{1}{2}\beta(t) \tilde{S}^h_{\zeta}(\tu(t), t).
\end{equation}
As discussed in \cite{su2022dual}, the DDIB is equivalent to a Schrödinger bridges problem and can also be interpreted as a Monge-Kantorovich optimal transport problem with an additional entropy regularization term. While conceptually related to traditional OT-based methods, the DDIB offers greater flexibility and adaptability in handling complex translation tasks.

\subsection{Enhanced DDIB}\label{sec:eddib}
In this subsection, we introduce the enhanced DDIB (EDDIB) method, which is based on a more general setting. Formally, let $\uul(t_1)$ be the translation using PF ODE driven by velocity field $\mTl_\xi$ from $t=0$ to $t=t_1$ with initial condition $u^l \sim p(u^l)$, that is 
\begin{equation}\label{eqn:ult1}
    \uul(t_1) = \text{ODEsolve}(u^l, 0, t_1; \mTl_\xi).
\end{equation}
The resulting $\textit{perturbed}$ distribution of $\uul(t_1)$ is denoted as $p(\uul(t_1))$. Similarly, let $\utu(t_2)$ be the translation using PF ODE driven by velocity field $\mTt_\zeta$ from $t=0$ to $t=t_2$ with initial condition $\tu \sim p(\tu)$, that is 
\begin{equation}\label{eqn:tut2}
\utu(t_2) = \text{ODEsolve}(\tu, 0, t_2; \mTt_\zeta).
\end{equation}
The resulting $\textit{perturbed}$ distribution of $\utu(t_2)$ is denoted as $p(\utu(t_2))$. The standard DDIB requires these two distributions align closely with the standard Gaussian distribution at $t_1=t_2=1$, i.e., $p(\uul(1)) \approx \mN(0, \mI) \approx p(\utu(1))$. However, achieving this alignment in practice can be computationally expensive. Note that the distributions $p(u^l(1))$ and $p(\uul(1))$ are different. The former is obtained from the forward diffusion process governed by the forward-time SDE \eqref{eqn:diff}. When the noise scheduling function $\beta(t)$ is chosen appropriately, we can expect $p(u^l(1)) \approx \mN(0, \mI)$. On the other hand, the translation from $u^l$ to $\uul(1)$ is deterministic and obtained by solving a PF ODE using a well-trained diffusion model $S^l_\xi$. Ensuring that the distribution $p(\uul(1))$ aligns closely with the standard Gaussian distribution requires a sufficiently large dataset $\{ u^l_i\}_{i=1}^N$, which may not always be available in practice. 

To resolve this issue, we propose the following EDDIB. We denote $\hat{u}(t_1, t_2)$ as the $\textit{translated LFLR}$ state using PF ODE driven by velocity field $\mTt$ from $t=t_2$ to $t=0$ with initial condition $\uul(t_1)$ from \eqref{eqn:ult1}, that is
\begin{equation}\label{eqn:recon}
\hat{u}^l(t_1, t_2) = \text{ODEsolve}(\uul(t_1), t_2, 0; \mTt_\zeta).
\end{equation}
The resulting $\textit{translated LFLR}$ distribution of $\hat{u}^l(t_1, t_2)$ is denoted as $p(\hat{u}^l(t_1, t_2))$. The performance of this unpaired translation task can be evaluated by measuring the distance, under a specific metric, between the translated distribution $p(\hat{u}^l(t_1, t_2))$ and the target distribution $p(\tu)$. Because the DDIB relies on a deterministic translation mechanism, instead of directly comparing $p(\hat{u}^l(t_1, t_2))$ and $p(\tu)$, we can examine the distance between two intermediate distributions. Specifically, both the translation from $p(\utu(t_2))$ to $p(\tu)$ and translation from $p(\uul(t_1))$ to $p(\hat{u}^l(t_1, t_2))$ are governed by the same PF ODE with the same velocity field $\mTt_\zeta$ over the same time interval, and this common governing mechanism allows us to evaluate the performance of unpaired translation task using hyperparameters $t_1$ and $t_2$ by studying the distance between $p(\uul(t_1))$ and $p(\utu(t_2))$ under a specified metric. We propose the following two propositions for performance evaluation, which also illustrate the advantages of the EDDIB. Moreover, we provide their proofs in \cref{sec:proof}.

\begin{proposition}\label{prop:1}
For any $t_1, t_2 \in (0, 1]$, the KL divergence between the translated LFLR distribution $p(\hat{u}^l(t_1, t_2))$ and the target HFLR distribution $p(\tu)$ equals the KL divergence between two perturbed distribution $p(\uul(t_1))$ and $p(\utu(t_2))$, that is
\[
D_{KL}(p(\hat{u}^l(t_1, t_2)) \| p(\tu)) = D_{KL}(p(\uul(t_1))\|p(\utu(t_2))).
\]
\end{proposition}
This proposition indicates that minimizing the KL divergence between the translated distribution $p(\hat{u}(t_1, t_2))$ and the target distribution $p(\tu)$ can be achieved by choosing hyperparameters $t_1$ and $t_2$ that minimize the KL divergence between $p(\uul(t_1))$ and $p(\utu(t_2))$. Note that the standard DDIB setup corresponds to the special case $t_1=t_2=1$. In contrast, the EDDIB permits flexibility by allowing $t_1$ and $t_2$ to vary, enabling a search for the optimal hyperparameter pair. Beyond KL divergence, we also leverage the stability of ODE flow in the Wasserstein-2 ($\mW$) distance, which motivates the second proposition.
\begin{proposition}\label{prop:2}
Assume $\mTt_\zeta(\tu(t), t)$ is $L_s$-Lipchitz continuous in $\tu(t)$, then for any $t_1, t_2 \in (0, 1]$, the $\mW$ distance between the translated distribution $p(\hat{u}^l(t_1, t_2))$ and the target distribution $p(\tu)$ is upper bounded by the $\mW$ distance between two perturbed distribution $p(\uul(t_1))$ and $p(\utu(t_2))$
\[
\mW(p(\hat{u}^l(t_1, t_2)), p(\tu)) \leq e^{L_s  t_2} \mW(p(\uul(t_1)), p(\utu(t_2))).
\]
\end{proposition}
This proposition shows that, unlike \cref{prop:1}, the upper bound now includes a coefficient dependent on $L_s$ and $t_2$. Although we lack prior knowledge of $L_s$, a key factor in minimizing the $\mW$ distance between $p(\hat{u}^l(t_1, t_2))$ and $ p(\tu)$ remains reducing the distance between $p(\uul(t_1))$ and $p(\utu(t_2))$. Both propositions therefore suggest that, to keep $p(\hat{u}^l(t_1, t_2))$ close to $ p(\tu)$, one should select optimal values $t^*_1$ and $t^*_2$ that minimize the distance between $p(\uul(t_1))$ and $p(\utu(t_2))$. This insight leads to \cref{alg:t1t2}, which identifies the optimal pair $t^*_1$ and $t^*_2$. In practice, however, relying solely on the $\mW$ and KL divergence does not necessarily yield the best performance. Instead, we experiment with various metrics and select the optimal one, as detailed in \cref{sec:numerical}.

\begin{algorithm}[htbp]
\caption{Selection of $t_1$ and $t_2$}
\label{alg:t1t2}
\begin{algorithmic}[1]
\REQUIRE Two datasets $\{ u^l_i \}_{i=1}^N$ and $\{ \tilde{u}^h_j \}_{j=1}^M$, two velocity fields $\mTl_\xi$ and $\mTt_\zeta$, number of steps $N_{t_1}$ and $N_{t_2}$, and a metric $\mathcal{M}$.
\STATE Initialize $d_{\min} \gets \infty$
\STATE Initialize $t^{*}_1 \gets 0$, $t^{*}_2 \gets 0$
\FOR{$p \gets 0$ to $N_{t_1}-1$}
    \STATE $t_1 \gets p \cdot \frac{1}{N_{t_1}-1}$.
    \FOR{$q \gets 0$ to $N_{t_2}-1$}
    \STATE $t_2 \gets q \cdot \frac{1}{N_{t_2}-1}$.
    \STATE Obtained $\{ \uul_i (t_1) \}_{i=1}^N$ via \eqref{eqn:ult1} using velocity $\mTl_\xi$.
    \STATE Obtained $\{ \utu_j (t_2) \}_{j=1}^M$ via \eqref{eqn:tut2} using velocity $\mTt_\zeta$.
    \STATE Compute $d = \mathcal{M}(\{ \uul_i (t_1) \}_{i=1}^N, \{ \utu_j (t_2) \}_{j=1}^M)$.
    \IF{$d < d_{\min}$}
        \STATE $d_{\min} \gets d$
        \STATE $t^{*}_1 \gets t_1$, $t^{*}_2 \gets t_2$.
    \ENDIF
    \ENDFOR
\ENDFOR
\RETURN Optimal $t^{*}_1$ and $t^{*}_2$.
\end{algorithmic}
\end{algorithm}
For brevity, we denote the translated data as $\hat{u}^{l,*}_i = \hat{u}^l_i(t_1^*, t_2^*)$ and we use $\mathcal{T}$ for this translation, that is $\hat{u}^{l,*} = \mathcal{T} u^l$. The entire process of tanslating $\datal$ to $\datath$ is summarized in \cref{alg:lr_trans}.

\begin{algorithm}[htbp]
\caption{Translation by EDDIB}
\label{alg:lr_trans}
\begin{algorithmic}[1]
\REQUIRE LFLR testing dataset $\dataltest$, two velocity fields $\mTl_\xi$ and $\mTt_\zeta$, number of steps $N_{t_1}$ and $N_{t_2}$, and a metric $\mathcal{M}$.
\STATE Obtain optimal $t^*_1, t^*_2$ from \cref{alg:t1t2}.
\STATE Obtain $\{ \uul_i(t_1^*) \}_{i=1}^Q$ using \eqref{eqn:ult1}.
\FOR{$i \gets 1$ to $Q$}
\STATE Compute intermediate state $\uul_i(t_1^*)=\text{ODEsolve}(u^l, 0, t^*_1; \mTl_\xi)$, see equation \eqref{eqn:ult1}.
\STATE Compute intermediate state $\hat{u}^{l,*}_i = \hat{u}_i(t_1^*, t_2^*)=\text{ODEsolve}(\uul_i(t_1^*), t^*_2, 0; \mTt_\zeta)$, see equation \eqref{eqn:recon}.
\ENDFOR
\RETURN Translated dataset $\{ \hat{u}^{l,*}_i \}_{i=1}^Q$.
\end{algorithmic}
\end{algorithm}

\subsection{Super-Resolution via SR3}\label{sec:sr}
In the SR step, we employ the cascaded SR3 model \cite{saharia2022image} to upscale the HFLR $\tu$ to HFHR $u^h$. Since the restriction operator $\mathcal{R}$ is user-specified and known, a paired dataset of HFLR and HFHR data, $\{ \tilde{u}_i^h, u_i^h \}_{i=1}^N$, can be generated from an HFHR dataset $\{u_i^h \}_{i=1}^N$. This paired dataset can then be used to train a conditional diffusion model $S_{\xi}(u^h(t), \tilde{u}^h, t)$, which is used for approximating the conditional distribution $p(u^h | \tilde{u}^h)$. The training process involves minimizing the following loss function:
\begin{equation}\label{eqn:loss_sr} L(\eta):= \mathbb{E}_{t \sim U(0, 1), u^h \sim p(u^h), \beps \sim \mN(0, \mI)} [\| \sigma(t) S_{\eta}(u^h(t), \tilde{u}^h, t) + \beps \|_2^2],
\end{equation}
as detailed in \cref{alg:uh}. 

With a well-trained conditional diffusion model $S_{\eta}(u^h(t), \tilde{u}^h, t)$ obtained from \cref{alg:uh}, a HFHR data $u^h$ corresponding to the given HFLR data $\tilde{u}^h$ can be generated using Predictor-Corrector (PC) samplers. As discussed in \cite{saharia2022image}, the super-resolution (SR) task with a large magnification factor can be split into a sequence of SR tasks with smaller magnification factors. This approach enables parallel training of simpler models, each requiring fewer parameters and less training effort. The training of each SR model follows the procedure outlined in \cref{alg:uh}.

\subsection{Neural Operator for dynamics}\label{sec:fno}
Neural operators, such as FNO \cite{li2020fourier} and DeepONet \cite{lu2021learning}, are widely used methods for learning dynamics directly from data. Consider an evolutionary PDE
\begin{equation}
    \begin{cases}
        \partial_t u(x, t) = \mL(u), \quad \quad (x, t) \in D \times (0, T] \\
        u(x, 0) = u_0(x), \quad \quad x \in D,
    \end{cases}
\end{equation}
where $\mL$ is a differential operator, $u_0(x)\in \mathcal{V}$ is the initial condition and $u(x, t) \in \mathcal{U}$ for $t>0$ is the solution trajectory. Here $D \subset \mathbb{R}^d$ is a bounded open set and $\mathcal{V}=\mathcal{V}(D;\mathbb{R}^d), ~\mathcal{U}=\mathcal{U}(D;\mathbb{R}^d)$ are two separable Banach spaces. Our objective is to approximate the solution operator $\mG: u_0(x) \mapsto u(x, t)$ using a neural network.

Note that the initial function $u_0(x)$ is defined on the spatial domain while the trajectory $u(x, t)$ is defined on the spatiotemporal domain. To handle the temporal dimension, two common approaches are typically employed. The first approach treats temporal $t$ as an independent variable alongside spatial variables, increasing the physical dimensionality of the problem. While theoretically sound, it requires a large number of snapshots to resolve transient dynamics, leading to prohibitive computational costs. In this work, we adopt the second strategy, discretizing the temporal domain into a fixed sequence of $n$ snapshots. These snapshots are treated as input channels, forming a trajectory $\bu := \{ u_0, u_1, \cdots, u_{n-1} \} \in \mathbb{R}^{n \times q \times q}$, where $u_0 \in \mathbb{R}^{q \times q}$ is the initial condition and $q$ denotes the size of mesh grid. This discretization enables efficient learning of the operator $\mG: u_0 \rightarrow \boldsymbol{u}$, which maps the initial state to the spatiotemporal trajectory. As shown in \cref{fig:diag_evo}, to enhance the LFLR simulation data, the dynamics can be learned at two resolution levels: the neural operator $\mG^h_{\phi}$ approximates the operator $\mG^h$ at HFHR level and $\tilde{\mG}_{\psi}$ as the approximation to operator $\tilde{\mG}^h$ at HFLR level. The training of the neural operator $\mG^h_{\phi}$ involves minimizing the loss function $L(\phi) = \mathbf{E}_{\bu^h \sim p(\bu^h)} (\| \mG_{\phi}(u^h_0) - \bu^h \|)$,
which is detailed in \cref{alg:fno}. Similarly, $\tilde{\mG}_{\psi}$ is trained via an analogous procedure, adapted to the HFLR simulation dataset.

\section{Diffusion-based Unpaired SR}
\subsection{Time-snapshot data}\label{subsec:snapshot}
For demonstration purposes, we focus on a scenario where the HR data is at resolution of $256 \times 256$ and LR data is at resolution of $32 \times 32$, although our method can be easily extended to more general settings. To handle the magnification factor of $8$, we partition the SR task into three smaller SR tasks, each with a magnification factor of 2. Specifically, we define three restriction operators $\mathcal{R}_1$, $\mathcal{R}_2$ and $\mathcal{R}$, which downsample the HFHR data to resolutions of $128 \times 128$, $64 \times 64$ and $32 \times 32$, respectively, to generate three datasets. By training three separate conditional diffusion models $S_{\eta_1}$, $S_{\eta_2}$ and $S_{\eta_3}$ using these three datasets in parallel and chain these models to perform cascaded SR, the LR data at $32 \times 32$ is progressively refined to $256 \times 256$.

The complete procedure for the time-snapshot problem consists of three stages. First, in the \textit{data preparation stage}, lower-resolution versions of the HFHR data are generated using user-specified restriction operators. Next, in the \textit{training stage}, two unconditional diffusion models are trained to facilitate the debiasing step at low resolution, while three of conditional diffusion models are trained with paired data to capture the relationships across different resolutions. At the \textit{inference stage}, the LFLR data are firstly translated using EDDIB with two well-trained unconditional diffusion models and then upsampled iteratively by three well-trained conditional diffusion models in a cascaded SR3 method. The implementation details are presented in \cref{alg:time_independent}.
\begin{algorithm}[htbp]
\caption{Unpaired SR for time snapshot data}
\label{alg:time_independent}
\begin{algorithmic}[1]
\REQUIRE An LFLR training dataset $\datal$ and an HFHR training dataset $\datah$, an LFLR testing dataset $\{ u^l_q\}_{q=1}^Q$, noise scheduling function $\alpha(t), \sigma(t)$, batch size $B$ and max iteration $Iter$.
\STATE \textbf{Data Preparation stage:}
\STATE Generate three lower resolution datasets from 
$\datah$ using $\mathcal{R}_1$, $\mathcal{R}_2$ and $\mathcal{R}$, resulting in $\datathl$, $\datathll$ and $\datath$.
\STATE \textbf{Training stage:}
\STATE Train two unconditional diffusion models $S^l_\xi$ and $\tilde{S}^h_\zeta$ on two datasets $\datal$ and $\datath$, respectively, via \cref{alg:utilde}.
\STATE Train three conditional diffusion models $S_{\eta_1}$, $S_{\eta_2}$ and $S_{\eta_3}$ on paired datasets $\pdatath$, $\pdatathl$ and $\pdatathll$, respectively, via \cref{alg:uh}.
\STATE \textbf{Inference stage:}
\STATE Translate the LFLR testing dataset $\dataltest$ using \cref{alg:lr_trans}, resulting in $\datahatl$.
\STATE Downscale $\datahatl$ to the final HR dataset $\datahath$ using cascaded SR3 via running \cref{alg:uh} iteratively based on $S_{\eta_1}$, $S_{\eta_2}$ and $S_{\eta_3}$.
\RETURN HFHR prediction $\datahath$.
\end{algorithmic}
\end{algorithm}

\subsection{Trajectory data}
In this subsection, we detail our methodology for enhancing LFLR trajectory data. As described in \cref{sec:problem_setting} and depicted in \cref{fig:diag_evo}, we consider three possible approaches: snapshot-wise refinement, HFLR dynamics learning, and HFHR dynamics learning. Due to its effectiveness in capturing fine-scale details, providing stable long-term predictions, and ensuring low computational overhead, we adopt the HFLR dynamics learning approach. The complete workflow consists of three main stages. 

In the \textit{data preparation} stage, we first transform trajectory datasets into snapshot datasets. Recall that we use bold symbols to denote trajectory data, with each trajectory represented as $\boldsymbol{u} := \{u^1, \dots, u^n\}$, where $n$ is the number of snapshots contained in each trajectory. Specifically, the LFLR trajectory dataset $\{ \bu^l_i\}_{i=1}^{N'}$ is decomposed into $\{ u^l_i\}_{i=1}^N$ with $N = nN'$, and similarly, the HFHR dataset $\{ \bu^h_j\}_{j=1}^{M'}$ yields $\{ u^h_j\}_{j=1}^{M}$ with $M = nM'$. We then generate corresponding lower-resolution versions of the HFHR datasets through user-defined restriction operators, producing paired HFLR and HFHR datasets.

Next, during the \textit{training} stage, we independently train two unconditional diffusion models on the LFLR and HFLR datasets to facilitate debiasing at low resolution, and a sequence of three conditional diffusion models on paired data to iteratively upsample the LR data; concurrently, an FNO is trained on the initial  state of LR trajectory data to model the system dynamics. Finally, at the \textit{inference} stage, the initial state of the LFLR trajectory data is first translated into an HFLR prediction using EDDIB with the two unconditional diffusion models. Subsequently, the FNO predicts the system dynamics at later time steps in the LR setting, and the cascaded SR3 is applied to upsample these predictions to high resolution. The complete procedure is detailed in \cref{alg:evo}.

\begin{algorithm}[ht]
\caption{Unpaired SR for Trajectory Data}
\label{alg:evo}
\begin{algorithmic}[1]
\REQUIRE LFLR trajectory training dataset $\edatal$ and HFHR trajectory training dataset $\edatah$, LFLR trajectory testing dataset $\edataltest$, noise scheduling function $\alpha(t), \sigma(t)$, three conditional diffusion models $S_{\eta_1}$, $S_{\eta_2}$ and $S_{\eta_3}$, batch size $B$ and max iteration $Iter$.
\STATE \textbf{Data Preparation stage:}
\STATE Convert the evolutionary datasets $\edatal$ and $\edatah$ into snapshots datasets $\datal$ and $\datah$, here $M=nM'$ and $N=nN'$.
\STATE Given three restriction operators $\mathcal{R}_1$, $\mathcal{R}_2$ and $\mathcal{R}$, generate lower resolution datasets $\datathl$, $\datathll$ and $\datath$.
\STATE Generate lower resolution evolutionary dataset $\edatath$ by applying restriction operator $\mathcal{R}$ to $\edatah$ for each trajectory snapshot wise.
\STATE \textbf{Training stage:}
\STATE Train two unconditional diffusion models $S^l_\xi$ and $\tilde{S}^h_\zeta$ on two datasets $\datal$ and $\datath$, respectively.
\STATE Train three conditional diffusion models $S_{\eta_1}$, $S_{\eta_2}$ and $S_{\eta_3}$ on paired datasets $\pdatath$, $\pdatathl$ and $\pdatathll$, respectively.
\STATE Train evolutionary models $\tilde{\mG}^h$ using $\edatath$.
\STATE \textbf{Inference stage:}
\STATE \textit{Initial state translation}: For the initial state $\{ u_i^{0, l} \}_{i=1}^{Q'}$ (the first snapshots in evolutionary dataset $\edataltest$), obtain $\{ \hat{u}^{*, 0, l}_i \}_{i=1}^{Q'}$ using \cref{alg:lr_trans}.
\STATE \textit{Dynamics prediction}: Apply the neural operator $\tilde{\mG}^h$ to $\{ \hat{u}^{*, 0, l}_i \}_{i=1}^{Q'}$, resulting in the prediction of HFLR evolutionary data, denoted as $\{ \hat{\mathbf{u}}^{l}_i \}_{i=1}^{Q'}$.
\STATE \textit{Super-resolution}: For each snapshot in $\{ \hat{\mathbf{u}}^{l}_i \}_{i=1}^{Q'}$, apply the cascaded SR3 ($S_{\eta_1}$, $S_{\eta_2}$ and $S_{\eta_3}$) to generate the final HFHR prediction $\{ \hat{\mathbf{u}}^{h}_i \}_{i=1}^{Q'}$.
\RETURN HFHR trajectory prediction $\{ \hat{\mathbf{u}}^{h}_i \}_{i=1}^{Q'}$.
\end{algorithmic}
\end{algorithm}

\section{Numerical Results}\label{sec:numerical}
In this section, we present numerical results demonstrating the effectiveness of the proposed method through three fluid dynamics problems. For snapshot problems, we consider the 2D Navier-Stokes equation, 2D Euler equation and {nonlinear water wave equation}. For the evolutionary problem, we focus on the 2D {Navier-Stokes equation}.

\subsection{Experimental settings}\label{sec:exp_setting}
For the time-snapshot data, we compare the performance of our proposed method with several baseline approaches, all of which use the cascaded SR3 model to perform SR from predicted HFLR data to predicted HFHR data. The primary difference lies in how each method realizes the translation task in the debiasing step. Each baseline model is named according to the specific debiasing method it employs:
\begin{itemize}
    \item \textbf{Direct}: This approach applies the SR3 model directly to LFLR data without a debiasing step.

    \item \textbf{OT}: Adapted from \cite{wan2023debias}, this method uses an OT map (implemented via \texttt{ott-jax}\cite{cuturi2013sinkhorn, cuturi2022optimal}) for debiasing at the LR stage, followed by cascaded SR3 for super-resolution (instead of using a single conditional diffusion model). 

    \item \textbf{NOT}: Employs neural optimal transport (NOT) as described in \cite{korotin2022neural} for debiasing at the LR stage, with subsequent cascaded SR3 for super-resolution.

    \item \textbf{SDEdit}: Uses SDEdit \cite{meng2021sdedit} to translate LFLR data into HFLR prediction, with subsequent cascaded SR3 for super-resolution.

    \item \textbf{EDDIB (Our method)}: Uses the EDDIB model to translate LFLR data into HFLR prediction, with subsequent cascaded SR3 for super-resolution.
\end{itemize}
For the trajectory data, we consider three approaches that differ in how they learn temporal dynamics, as described in \cref{sec:problem_setting}. For simplicity, we refer to the snapshot-wise enhancement, HFLR dynamics learning, and HFHR dynamics methods as Method I, II, and III, respectively. The details for producing the HFHR prediction at $t=T$ are as follows:
\begin{itemize}
    \item \textbf{Method I (LR solver $+\mathcal{T}+\mathcal{S}$)}: Starting with initial LFLR data $u^{0, l}$, the LR solver produces an LFLR prediction $u^{T, l}$ at $t=T$. Subsequently, EDDIB is applied for translation, and cascaded SR3 is used for super-resolution, yielding the HFHR prediction at $t=T$.

    \item \textbf{Method II ( $\mathcal{T}+\tilde{\mathcal{G}}^h+\mathcal{S}$ )}:  
    EDDIB is first used to translate the initial LFLR data \( u^{0, l} \) into HFLR prediction \( \tilde{u}^{0, h} \). An FNO that learns coarse-scale dynamics then generates an HFLR prediction \( \tilde{u}^{T, h} \) at \( t = T \), which is subsequently refined using cascaded SR3 to obtain the HFHR prediction at \( t = T \).

    \item \textbf{Method III ( $\mathcal{T}+\mathcal{S}+\mathcal{G}^h$)}:  
    In this approach, EDDIB and cascaded SR3 are employed to enhance the initial LFLR data \( u^{0, l} \) into an HFHR prediction at \( t = 0 \). An FNO that learns fine-scale dynamics then produces the HFHR prediction at \( t = T \).
\end{itemize}

\textbf{Data generation.}\label{sec:data_gen} To generate the unpaired LFLR and HFHR training datasets, $\datal$ and $\datah$, and  the paired LFLR and HFHR evaluation datasets, $\dataltest$ and $\datahtest$, for the 2D Navier-Stokes equation and 2D Euler equation, we proceed as follows. First, we randomly sample $N+M+Q$ initial conditions from a certain random field discretized on a $256 \times 256$ spatial grid. The first $N$ initial conditions are downsampled to a resolution of $32 \times 32$ and use them as inputs to a LR solver, running for a temporal duration of $T$ to construct the LFLR dataset $\datal$. For the next $M$ initial conditions, we apply an HR solver for the same duration $T$ to generate HFHR dataset $\datah$. For the last $Q$ initial conditions, we first downsample them to a resolution of $32 \times 32$ and use them as inputs to an LR solver, running for a temporal duration of $T$ to construct the LFLR dataset $\datal$. Simultaneously, we use the original HR initial conditions as inputs to the HR solver running for same duration, which generates the paired LFLR-HFHR dataset $\datatest$ for evaluation.

To generate the dataset for the nonlinear water wave evolution, we first initialize a wave field on a $256\times 256$ grid based on a realistic ocean wave energy spectrum. The wave field is then evolved for $1000$ s using a wave solver. The solver is described in detail in \cref{sec:wave}. We then extract the wave fields at $t=0, 50, 150, \ldots 950\,\text{s}$ and downsample them to the $32\times 32$ resolution. These downsampled snapshots are used as the initial conditions to the wave solver, which are run for a duration to construct the LFLR dataset. It should be noted that although the extracted snapshots all evolve from one initial condition, they can be considered as independent realizations of ocean wave fields because the nonlinear wave evolution can make the wave fields incoherent after $50\,\text{s}$. 

\textbf{Metrics.} In this paper, we consider several metrics to quantify the distance between two empirical distributions. These metrics will be used for evaluating the quality of HFHR prediction dataset $\datahath$ obtained from our proposed method and the reference HFHR dataset $\{ u^h_i \}_{i=1}^Q$. First, we consider unweighted mean energy log ratio (MELRu) and weighted mean energy log ratio (MELRw). Additionally, we also consider the Maximum Mean Discrepancy (MMD), Relative Mean Square Error (RMSE), Wasserstein-2 distance ($\mW$) and Total Variation Distance (TVD), the details are provided in \cref{sec:metrics}.

\textbf{Implementation and hyperparameter settings.} We use bicubic interpolation for all restriction operators $\mathcal{R}_1$, $\mathcal{R}_2$ and $\mathcal{R}$ (as applied in  \cref{alg:time_independent}, ~\cref{alg:evo}). For all the scored based diffusion model $\mS_{\xi}$, $\mS_{\zeta}$, $S_{\eta_1}$, $S_{\eta_2}$ and $S_{\eta_3}$, we adopt the UNet architecture as described in \cite{song2020score}. Details on the selection of the noise scheduling function and the hyperparameters can be found in \cref{sec:archi_hyper} in Appendix. To implement the $\text{ODEsolve}$ in \eqref{eqn:pfode_notation} and related equations, we use the RK45 scheme with a stopping criterion of $rtol=1e-5$ and $atol=1e-5$. The number of searching steps in \cref{alg:t1t2} is set to $N_{t_1}=N_{t_2}=10$. For the evolutionary problem, we adopt the FNO \cite{li2020fourier} as the neural operator. The details of both $\tilde{\mathcal{G}}^h$ and $\mathcal{G}^h$, along with the architecture and the selection of hyperparameters for other baseline models, such as NOT and SDEdit, are presented in \cref{sec:archi_hyper}.

\subsection{2D Navier-Stokes equation}\label{sec:ns}
We consider the vorticity form of the 2D Navier-Stokes equation with Kolmogorov force:
\begin{equation}\label{eqn:ns}
    \begin{cases} 
        \partial_t \omega(t,\bx) + \mathbf{u}(t,\bx) \cdot \nabla \omega(t,\bx) - \nu \nabla^2 \omega(t,\bx) = f(\bx), \quad \bx \in [0, 2\pi]^2, t \in [0,T]\\
        \nabla \cdot \mathbf{u}(t,\bx) = 0 , \quad \bx \in [0, 2\pi]^2, t \in [0,T] \\
        w(0, \bx) = w_0(\bx), \quad \bx \in [0, 2\pi]^2,
    \end{cases}
\end{equation}
where $\bx = (x_1, x_2)$, $\omega(t,\bx)$ represents the vorticity, and the velocity field is given by $\mathbf{u}(t,\bx)=\psi_y - \psi_x$, where $\omega=-\nabla^2 \psi$. The Kolmogorov force is defined as $f(\bx) =  \sin(k_0 x_1)$. In this example, we set the wavenumber $k_0 = 4$ and viscosity $\nu = 10^{-3}$. To generate initial conditions, we first generate samples from a log-normal distribution on a uniform $256 \times 256$ mesh grid defined over spatial domain $\bx \in [0, 2\pi]^2$. We run HR and LR solvers on these initial conditions for $T=1$ to obtain HFHR and LFLR datasets, respectively. In this example, both the LR and HR solvers use the finite-volume method (FVM) implemented in \texttt{jax-cfd} \cite{Dresdner2022-Spectral-ML,kochkov2021machine}. For training purposes, we use $N=4000$ LFLR data and $M=4000$ HFHR data, and use $Q=100$ LFLR-HFHR paired data for testing.

\textbf{Ablation study.} We begin by benchmarking the performance of several approaches for the debiasing step, that is to translate LFLR data to HFLR data. In addition to NOT and SDEdit, we evaluate the vanilla DDIB and our EDDIB using the three metrics described in \cref{alg:t1t2}. We denote these approaches as DDIB, EDDIB+$\mW$, EDDIB+MMD and EDDIB+MELRw. \cref{tab:comp} presents the measured distances between the translated distribution $p(\hat{u}^{*,l})$ and the target HFLR distribution $p(\tu)$ across various metrics for each method. The results clearly demonstrate that DDIB-based models offer significant advantages for the translation task. Although the MMD metric does not differentiate the methods substantially, the other metrics reveal that DDIB-based models perform considerably better. Notably, EDDIB+MELRw delivers the best overall performance, 
we thus adopt MELRw for all subsequent experiments and omit the ``MELRw'' suffix; that is, the EDDIB model hereafter refers to the EDDIB+MELRw variant.

\begin{table}[htbp]
    \centering
    \begin{tabular}{|c|c|c|c|c|c|}
    \hline
          Config & $\text{MMD}$ & $\text{MELRu}$ & $\text{MELRw}$ & $\mW$ & $\text{D}_{KL}$\\
    \hline
    \hline
    LFLR & \textbf{0.020} & 1.247 & 0.527 & 0.061 & 0.036 \\
    OT & 0.027  & 0.365  & 0.334  & 0.105  & 0.079 \\
    NOT & 0.021 & 0.754 & 0.236 & 0.063 & 0.051 \\
    SDEdit & \textbf{0.020} & 0.292 & 0.180 & 0.112 & 0.070 \\
    \hline
    \hline
    DDIB & \textbf{0.020} & 0.298 & 0.117 & 0.065 & 0.045\\
    EDDIB+$\mW$ & \textbf{0.020} & 0.941 & 0.436 & 0.061 & 0.037\\
    EDDIB+MMD & \textbf{0.020} & 0.207 & \textbf{0.062} & 0.066 & 0.039 \\
    EDDIB+MELRw & \textbf{0.020} & \textbf{0.146} & 0.109 & \textbf{0.059} & \textbf{0.034} \\
    \hline
    \end{tabular}
    \caption{Performance comparison of different debiasing methods for translating LFLR data to match the target HFLR distribution across various metrics. The table lists several distance metrics—MMD, MELRu, MELRw, Wasserstein distance ($\mW$), and KL divergence ($\text{D}_{KL}$). The variants EDDIB+$\mW$, EDDIB+MMD, and EDDIB+MELRw denote the EDDIB employing $\mW$, MMD, MELRw, respectively, as described in \cref{alg:t1t2}.}
    \label{tab:comp}
\end{table}
In \cref{fig:LR_comp}, we compare the predictions of the translated LFLR data produced by various approaches. The NOT method lacks the desired fine scale details, and we believe this issue might be resolved by incorporating more sophisticated regularization; however, that is beyond the scope of this paper. Although the SDEdit method recovers some of these fine details, it does so at the expense of altering the large scale structure due to an inherent trade off. In contrast, the EDDIB predictions simultaneously preserve the large-scale structure and generate the desired fine-scale details.

\begin{figure}[ht]
    \centering
    \includegraphics[width=0.9\textwidth]{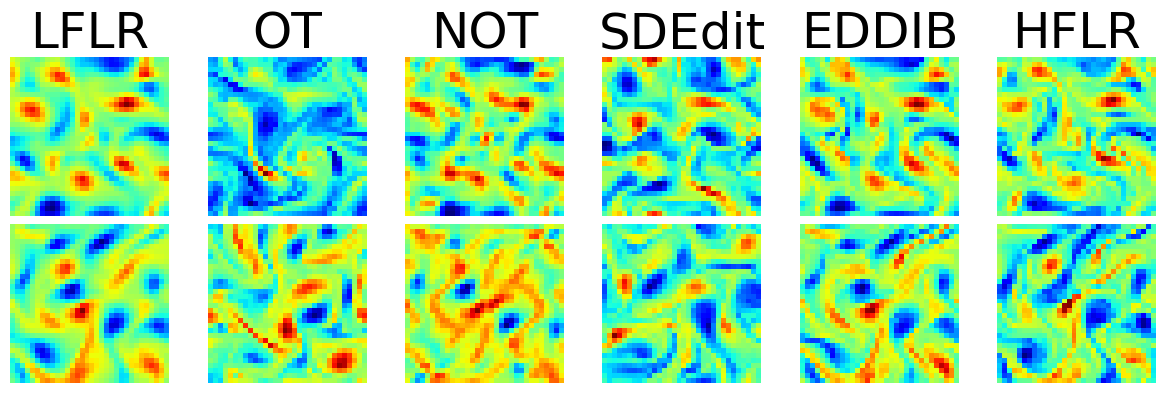}
    \caption{Comparison of the translated LFLR obtained from various methods. The first and the second rows presents two instances.}
    \label{fig:LR_comp}
\end{figure}

\textbf{Results of SR.}
\cref{fig:ns} presents a comparison of unpaired SR results from five baseline approaches alongside the reference HFHR data. The top and bottom rows display two instances of predictions from various baseline models and the reference. It is evident that both the Direct SR and NOT methods yield predictions lacking the desired fine-scale details. In contrast, the OT and SDEdit recover small-scale details similar to the reference, though they alter the large-scale structure. The EDDIB method produces predictions that preserve the large-scale structure while generating the desired high-frequency details, as illustrated in the bottom-right sub-figure. 

Additionally, we provide boxplots of various distance metrics between each prediction and the reference. The bottom-left sub-figure shows that the OT and SDEdit methods exhibit significantly higher RMSE and TVD values compared to the others, indicating that they fail to preserve large-scale structures. Meanwhile, for MELRu and MELRw, the NOT method yields larger values. In the bottom-right sub-figure, the log ratio of the energy spectra for each baseline is plotted relative to the reference, revealing that both SDEdit and EDDIB better align with the reference in the wavenumber domain. Notably, EDDIB achieves low RMSE and TVD (thus maintaining large-scale structures) while also recovering fine-scale details (reflected by small MELRu and MELRw values and a low log ratio of the energy spectra).
\begin{figure}[ht]
    \centering
    \includegraphics[width=0.9\textwidth]{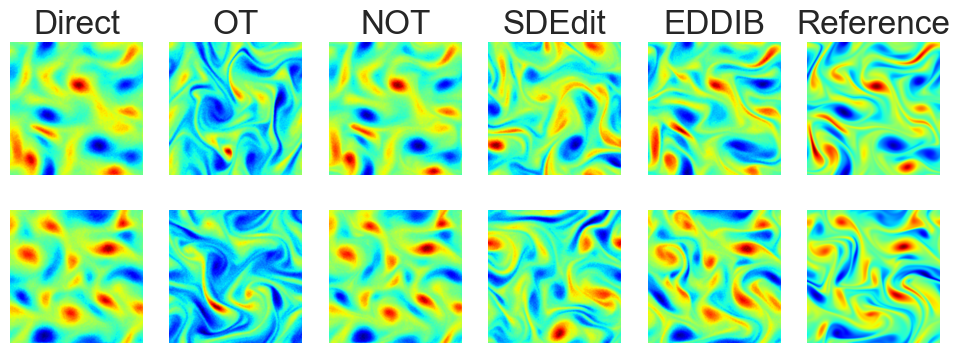}
    \includegraphics[width=0.45\textwidth]{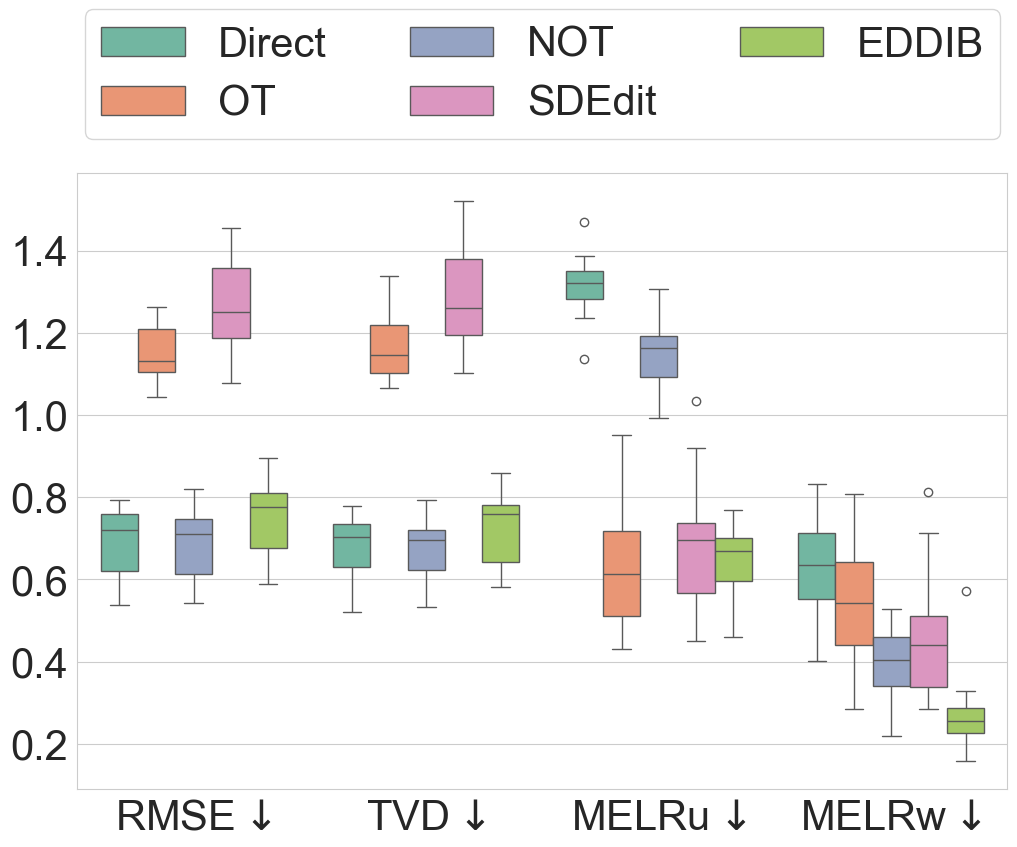}
    \includegraphics[width=0.45\textwidth]{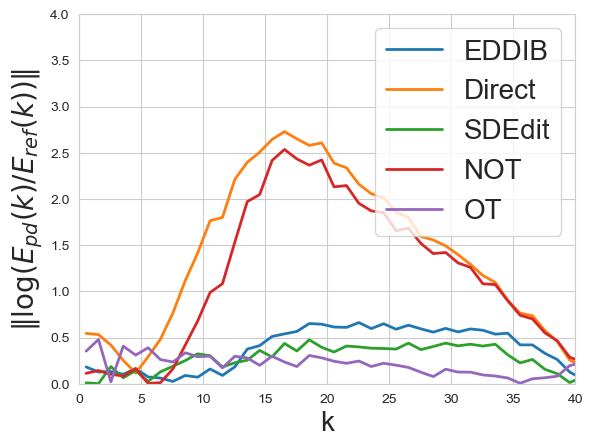}
    \caption{2D Navier–Stokes equation: SR results from five baseline methods compared to the reference. The top two rows show the predictions for two different LFLR data. The bottom-left panel displays the boxplots comparing four distance metrics between each prediction and the reference. The bottom-right plot displays the log ratio of the predicted energy spectrum relative to the reference, illustrating that EDDIB most effectively preserves both large-scale structures and fine-scale details.}
    \label{fig:ns}
\end{figure}

\subsection{2D Euler equation with shocks} In this part, we consider 2D Euler equations
\begin{equation}
    \begin{cases} 
    \rho_t + (\rho u)_x + (\rho v)_y = 0, \\
    (\rho u)_t + (\rho u^2 + p)_x + (\rho uv)_y = 0, \\
    (\rho v)_t + (\rho uv)_x + (\rho v^2 + p)_y = 0, \\
    E_t + (u (E + p) )_x + (v (E + p))_y = 0. \\
    \end{cases}
\end{equation}
Here, $\rho(x, y, t)$ is the density, $u(x,y,t)$ and $v(x,y,t)$ are the velocity field in $x$ and $y$ direction respectively, and $E(x,y,t)$ is the total energy. The spatial domain is $(x,y)\in [0,1]^2$ with periodic boundary condition. The initial condition is adapted from the 2D Riemann problem from \cite{liska2003comparison}, which is defined as 
\[
(\rho,u,v,p)(x, y, 0)=(\rho_0,u_0,v_0,p_0) + 0.15 \mathbf{z}
\]
where
\begin{equation}
(\rho_0,u_0,v_0,p_0) \;=\;
\begin{cases}
(0.5323,\,1.206,\,0,\,0.3), 
& \text{if } x \le 0.5,\;y \le 0.5,\\[6pt]
(1.5,\,0,\,0,\,1.5), 
& \text{if } x > 0.5,\;y \le 0.5,\\[6pt]
(0.138,\,1.206,\,1.206,\,0.029), 
& \text{if } x \le 0.5,\;y > 0.5,\\[6pt]
(0.5323,\,0,\,1.206,\,0.3), 
& \text{if } x > 0.5,\;y > 0.5.
\end{cases}
\end{equation}
and the vector $\mathbf{z}:=(z_1, z_2, z_3, z_4)$ whose components are i.i.d sampled from the standard Gaussian distribution i.e. $z_i \sim \mN(0,1)$ for $i=1,2,3,4$. To close the equations, we use the following equation of state for ideal gas with $\gamma=1.4$: 
\[
E = \frac{1}{2} \rho u^2 + \frac{p}{\gamma - 1}.
\]

\textbf{Results of unpaired SR.} In this example, we use $N=4000$ LFLR data and $M=4000$ HFHR data for training, and use $Q=100$ paired LFLR-HFHR data for testing. The results of the comparison between different approaches are presented in \cref{fig:euler}. From the top two rows, we observe that EDDIB not only preserves the large-scale shock profile and discontinuity location but also recovers fine-scale eddies. In contrast, Direct and NOT noticeably lack small-scale structures, while SDEdit and OT fail to maintain the large-scale shock profile. The bottom two subplots further confirm these observations. In the bottom-left subplot, EDDIB achieves the smallest MELRu and MELRw, demonstrating its effectiveness in reconstructing multi-scale shock wave structures. In the bottom-right subplot, although both EDDIB and SDEdit maintain relatively low spectral errors across most wavenumbers, the RMSE and TVD values in the bottom-left subplot reveal that SDEdit and OT have significantly larger errors than EDDIB—indicating their inability to preserve the large-scale structure.

\begin{figure}[htbp]
    \centering
    \includegraphics[width=0.9\textwidth]{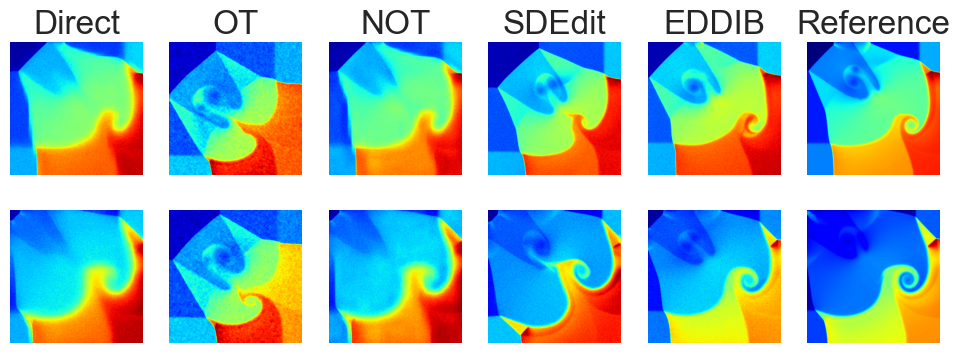}
    \includegraphics[width=0.45\textwidth]{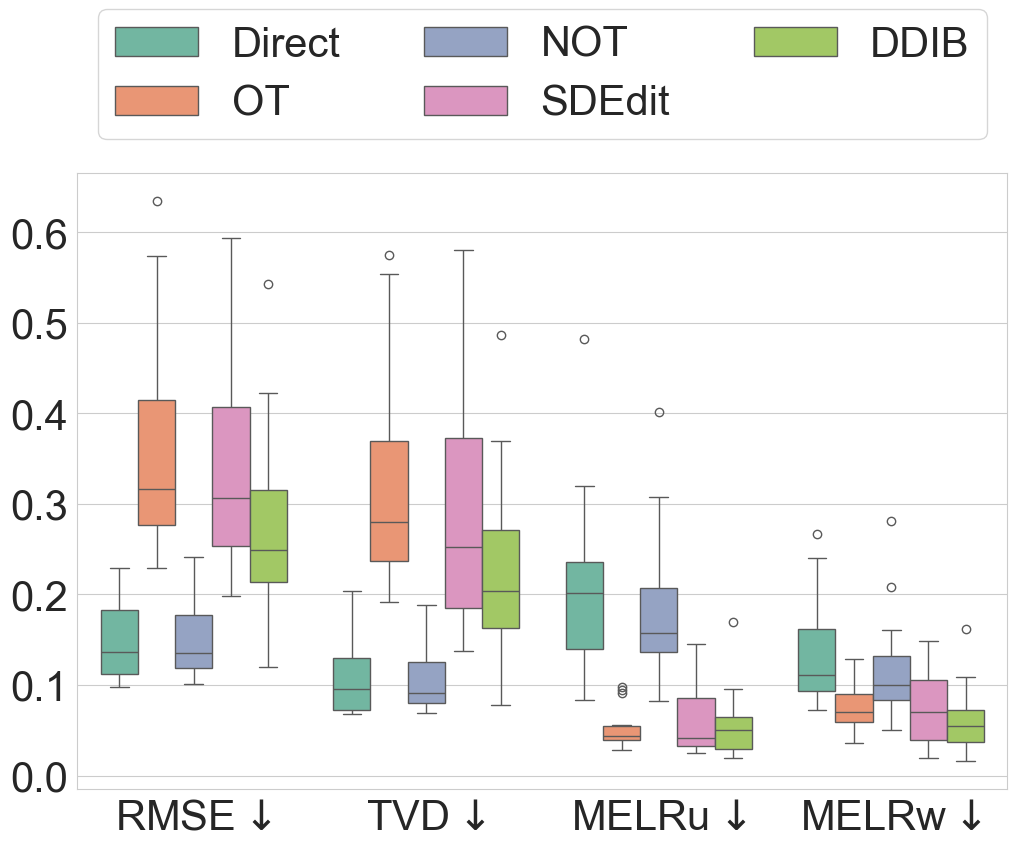}
    \includegraphics[width=0.45\textwidth]{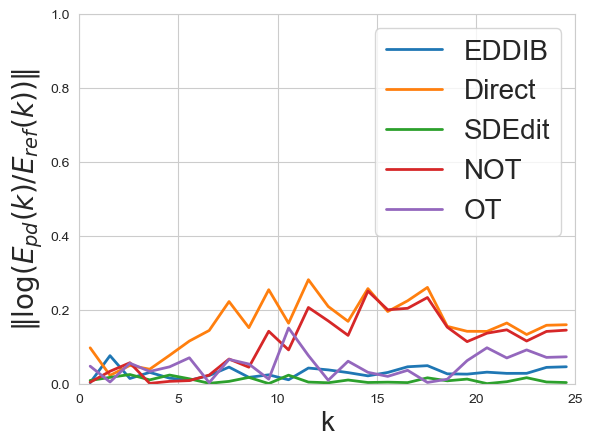}
    \caption{2D Euler equation: SR results from five baseline methods compared to the reference. The top two rows show the predictions for two different LFLR data. The bottom-left panel displays the boxplots comparing four distance metrics between each prediction and the reference. The bottom-right plot displays the log ratio of the predicted energy spectrum relative to the reference.}
    \label{fig:euler}
\end{figure}

\subsection{2D nonlinear water waves}\label{sec:wave} 
In this subsection, we focus on snapshot datasets for a nonlinear water wave system. A wave system can be described by the surface elevation $\eta(x, y, t)$ and the velocity potential $\phi(x, y, z, t)$. The evolution equations of the waves are given by \cite{zakharov1968stability}
\begin{equation}
\begin{cases}
\eta_t+ \phi^S_x \eta_x +\phi^S_y \eta_y -\left(1+\eta_x^2 +\eta_y^2\right) {[\phi_z]}_{z=\eta} = 0, \\
\phi^S_t+g\eta+\frac{1}{2} \left( {(\phi^S_x)^2 + (\phi^S_y)^2} \right) -\frac{1}{2} \left( {1+\eta_x^2 +\eta_y^2} \right) {[\phi^2_z]}_{z=\eta} = 0,
\end{cases}
\end{equation}
where $\phi^S(x,y,t)={[\phi(x,y,z,t)]}_{z=\eta}$ denotes the velocity potential at the wave surface $z=\eta(x, y)$ and $g$ is the gravitational acceleration. The above equations are solved by the high-order spectral method~\cite{dommermuth1987high}, a numerical scheme widely used for simulating ocean waves \cite{hao2020direct,xuan2024effect}. The spatial domain is $(x,y) \in [0, 100]^2\,\text{m}^2$. The initial condition is generated based on the JONSWAP spectrum~\cite{hasselmann1973}, a spectral energy distribution of realistic wind-generated ocean waves across different wavenumbers. The directional wave spectrum is given by
\begin{equation}\label{3-1}
\begin{aligned}
    S(k,\theta)=\frac{\alpha_p}{\sqrt{k^5g}}\exp \left[-\frac{5}{4}{\left(\frac{k}{k_p}\right)^{-2}}\right]\gamma^{\exp\left[-(\sqrt{k}-\sqrt{k_p})^2/(2\sigma^2k_p)\right]}\cos^2(\theta), \\
\end{aligned}
\end{equation}
where $\alpha_p=1.076\times 10^{-2}$ is the Phillips parameter, $k_p=0.25\,\text{m}^{-1}$ is the peak wavenumber, \(g=9.8 \text{ m\,s}^{-2}\), $\gamma=3.3$ is the peak enhancement factor, $\sigma=0.07$ for \(k\le k_p\) and \(\sigma=0.09\) for \(k> k_p\), and \(\theta\in[-\pi/2,\pi/2]\) is the wave direction.
The initial wave field is constructed as a superposition of linear wave components of varying wavelengths with amplitudes determined by the energy density \eqref{3-1}.
The phases of the wave components are randomly assigned to provide a realization of random waves. 

\textbf{Results of SR.} In this example, we use $N=4000$ LFLR data and $M=4000$ HFHR data for training, and use $Q=100$ paired LFLR-HFHR data for testing. The results of the comparison between different approaches are presented in \cref{fig:wave}. From the top two rows, we observe that EDDIB effectively retains the large-scale surface wave patterns while also reconstructing fine-scale roughness associated with short waves. In contrast, Direct and NOT exhibit a lack of small-scale features, and OT and SDEdit struggle to maintain the large-scale wave profiles. The bottom two plots reinforce these observations. In the bottom-left plot, EDDIB achieves the smallest MELRu and MELRw, indicating its ability to capture multi-scale wave structures. Meanwhile, altough the bottom-right plot shows that OT and SDEdit achieve similarly low spectral errors as EDDIB at most wavenumbers, the bottom-left plot reveals that OT and SDEdit have substantially higher RMSE and TVD than EDDIT, underscoring their difficulty in preserving large-scale structures.

\begin{figure}[htbp]
    \centering
    \includegraphics[width=0.9\textwidth]{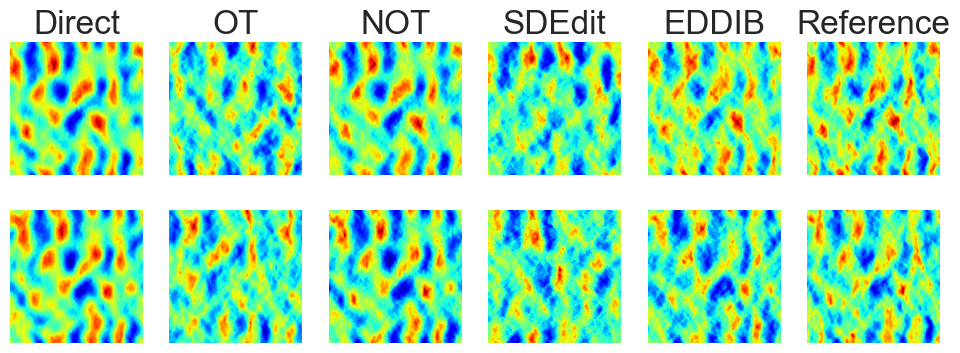}
    \includegraphics[width=0.45\textwidth]{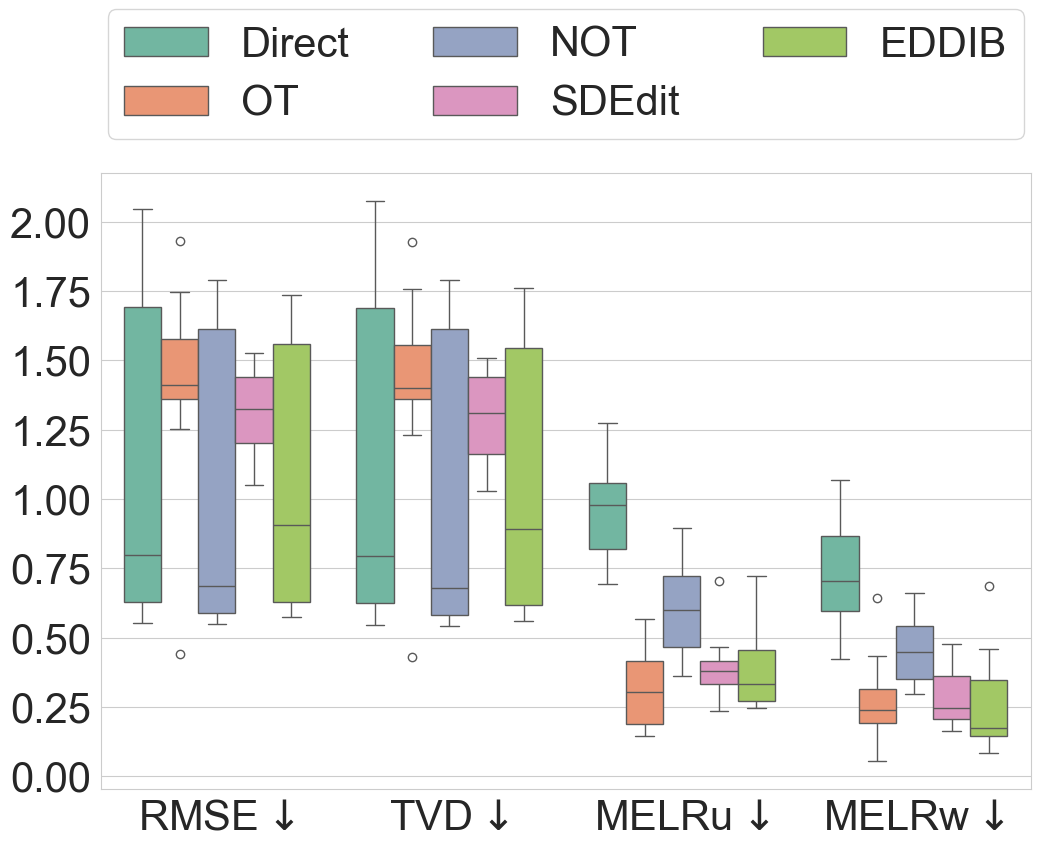}
    \includegraphics[width=0.45\textwidth]{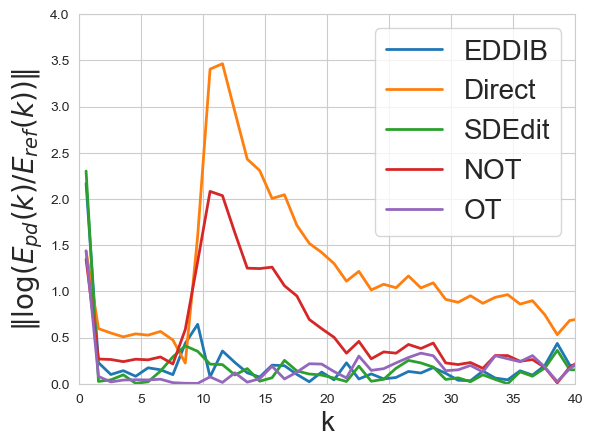}
    \caption{Surface elevation of nonlinear water waves: SR results from four baseline methods (Direct, OT, NOT, SDEdit, EDDIB) compared to the reference. The top two rows show the predictions for two different LFLR data. The bottom-left panel displays the boxplots comparing four distance metrics (RMSE, TVD, MELRu, MELRw) between each prediction and the reference. The bottom-right plot displays the log ratio of the predicted energy spectrum relative to the reference.}
    \label{fig:wave}
\end{figure}

\subsection{2D Navier-Stokes equation for trajectory data}
In this subsection, we consider the trajectory setting for the 2D Navier-Stokes equation \eqref{eqn:ns}. We set the number of snapshots in each trajectory to $n=5$ and the time step size to $\Delta t = 0.2$. The initial conditions are generated as described in \cref{sec:ns}. Based on previous numerical experiments, we observe that EDDIB effectively preserves large-scale structures while recovering the desired small-scale details. Consequently, we employ EDDIB for the translation task in the debiasing step across all three approaches for the trajectory problem, as detailed in \cref{sec:exp_setting}. To obtain the HFHR prediction at $T=1$, we consider three previously defined methods: Method I (snapshot-wise enhancement), Method II (HFLR dynamics learning), and Method III (HFHR dynamics learning), as introduced in \cref{sec:exp_setting}.

\textbf{Results of SR of trajectory data.} In this example, we use $N=1000$ LFLR trajectory data and $M=1000$ HFHR trajectory data for training, and use $Q=100$ paired LFLR-HFHR trajectory data for testing. We present the unpaired SR results at $T=1$ of three baseline methods for trajectory datasets in \cref{fig:ns_evo}. The top two rows show snapshots at $T=1$ for two different initial conditions. Among these three methods,  Method II appears to best capture both large-scale flow features and finer eddy structures. The bottom-left boxplot compares four error metrics (RMSE, TVD, MELRu, MELRw) across all test samples at $T=1$, where Method II achieves the best overall performance. In the bottom-right panel, the logarithmic ratio of the predicted energy spectrum to the reference solution shows that Method II remains closer to the reference across most wavenumbers.

Moreover, we present one example of the predictions at $t=0.2,0.4, 0.6,0.8, 1$ of three methods in \cref{fig:NS_evo_2}. Method I produces prediction at each snapshot with desired fine scale detail but the large scale flow pattern deviates from the reference solution along time. Method III performs well initially but deteriorates at later times because learning dynamics at a $256\times256$ resolution with only $M=1000$ HFHR trajectory datasets proves insufficient. In contrast, Method II benefits from learning dynamics at a $32\times32$ resolution, where $M=1000$ is adequate, leading to more accurate predictions throughout the entire trajectory.

\begin{figure}[htbp]
    \centering
    \includegraphics[width=0.6\textwidth]{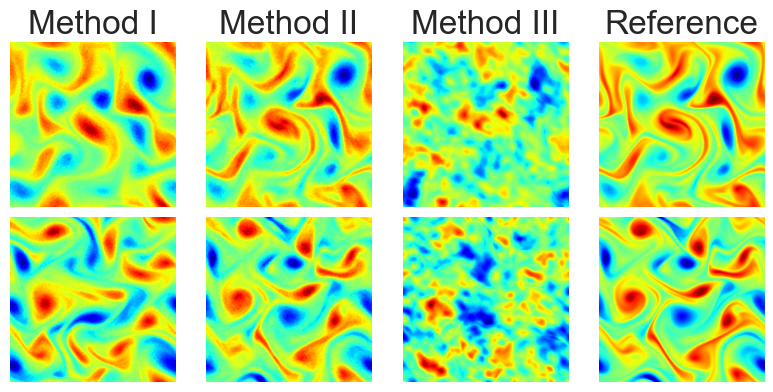}
    \includegraphics[width=0.4\textwidth]{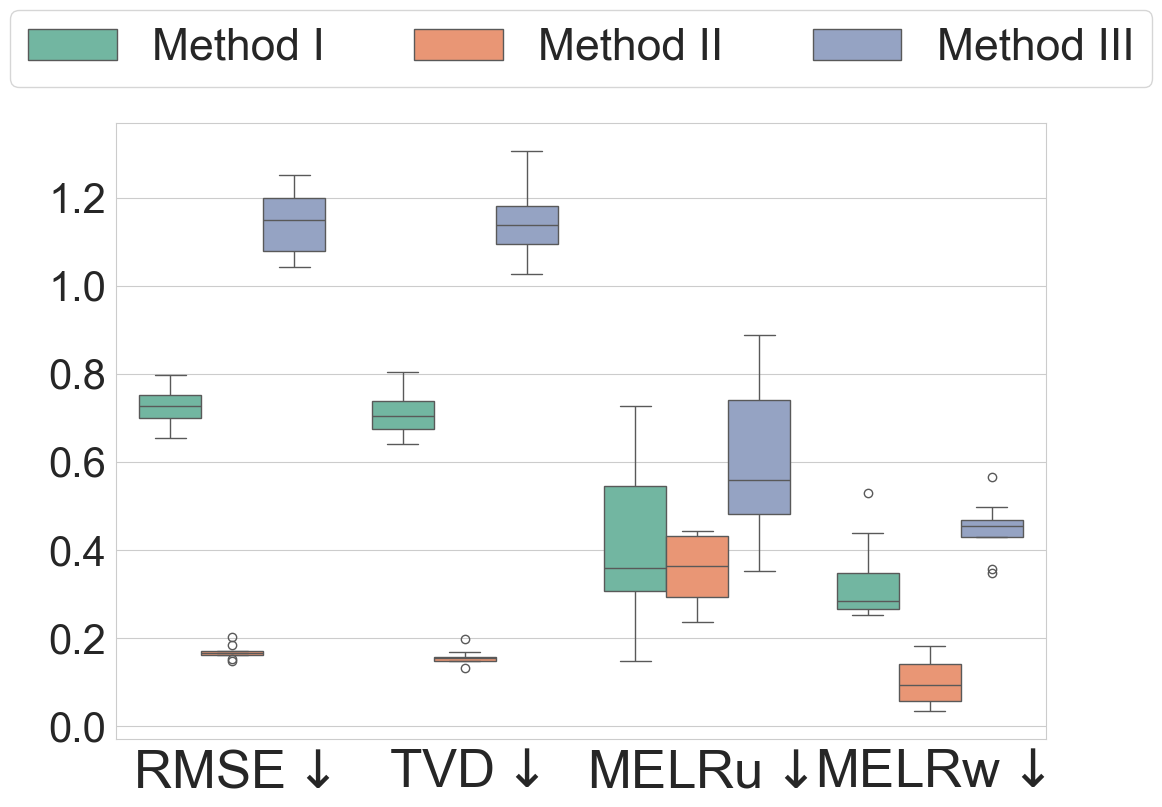}
    \includegraphics[width=0.4\textwidth]{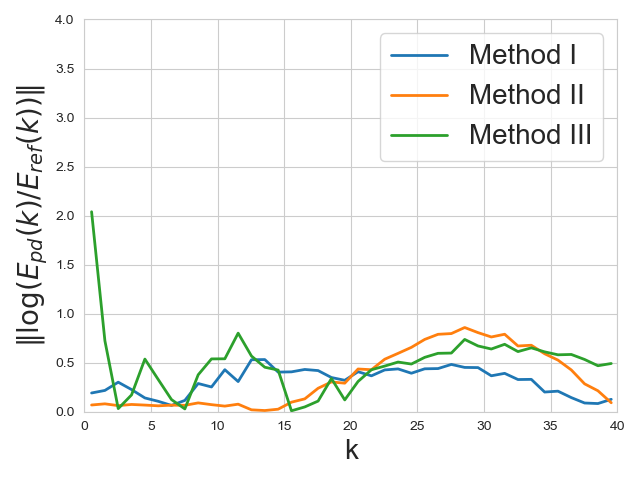}
    \caption{2D Navier–Stokes equation at $T=1$. SR results from three baseline methods compared to the reference. The top two rows show the predictions at final time $T=1$ for two different initial conditions. The bottom-left panel displays the boxplots comparing four distance metrics (RMSE, TVD, MELRu, MELRw) between each prediction at $T=1$ and the reference. The bottom-right plot displays the log ratio of the predicted energy spectrum relative to the reference.}
    \label{fig:ns_evo}
\end{figure}

\begin{figure}[htbp]
    \centering
    \includegraphics[width=0.9\textwidth]{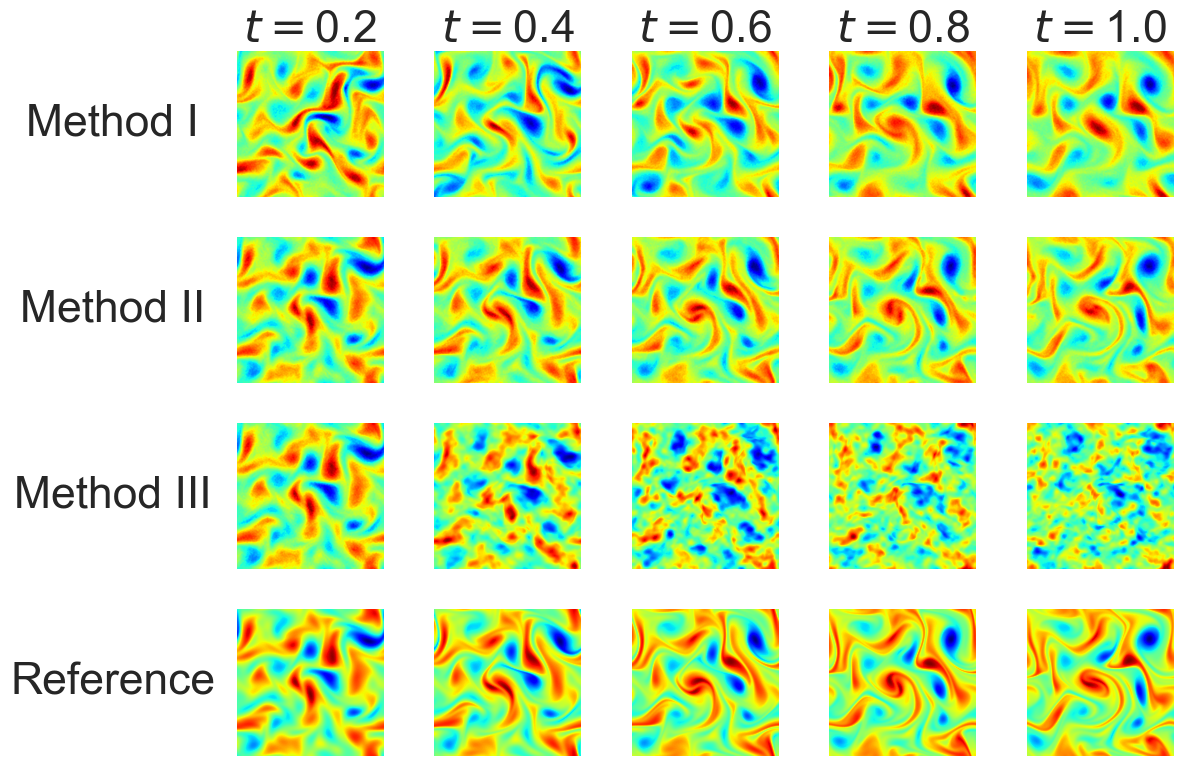}
    \caption{2D Navier–Stokes equation: prediction of trajectories obtained from various approaches.}
    \label{fig:NS_evo_2}
\end{figure}

\section*{Acknowledgment} 
We thank Zhong Yi Wan and Ricardo Baptista for their guidance on implementing the optimal transport methods detailed in their work \cite{wan2023debias}, which we used for comparison with our approach.

\bibliographystyle{siamplain}
\bibliography{references}

\appendix
\section{Notation Table}\label{sec:notation}
In this part, we summarize the notation used throughout the paper their description.
\begin{table}[H]
    \centering
    \begin{tabular}{c c}
    \hline
      Notation   &   Description \\
    \hline
    $u^h$&   High-fidelity high-resolution (HFHR) data \\
     $\tu$&  High-fidelity low-resolution (HFLR) data\\
     $u^l$& Low-fidelity low-resolution (LFLR) data\\
     $\bu^h$&  High-fidelity high-resolution (HFHR) evolutionary data \\
     $\tilde{\bu}^h$&  High-fidelity low-resolution (HFLR) evolutionary data\\
     $\bu^l$& Low-fidelity low-resolution (LFLR) evolutionary data\\
     $u^l(t)$ & Perturbed LFLR data using forward time SDE\\
     $\tu(t)$ & Perturbed HFLR data using forward time SDE\\
     $\beta(t)$ & Noise scheduling function\\
     $\alpha(t)$ & $\alpha(t) = e^{-\frac{1}{2}\int_0^t \beta(s) \rd s}$\\
     $\sigma(t)$ & $\sigma^2(t) = 1- \alpha^2(t)$ \\
     $S_{\zeta}(\tu(t), t)$ & Diffusion model used to approximate $p(\tu)$  \\
     $\mTl_\xi$ & Velocity field for PF ODE using $S_{\xi}$, see \eqref{eqn:velo1}  \\
     $\mTt_\zeta$ & Velocity field for PF ODE using $S_{\zeta}$, see \eqref{eqn:velo2}  \\
     $\uul(t)$ & Perturbed LFLR data using PF ODE with velocity field $\mTt_\xi$\\
     $\utu(t)$ & Perturbed HFLR data using PF ODE with velocity field $\mTt_\zeta$\\
     $\hat{u}^l(t_1, t_2)$ & Translated LFLR data using \eqref{eqn:recon} with $t_1, t_2$.\\
     $\hat{u}^{*,l}$ & \makecell{Translated LFLR data  with optimal $t^*_1, t^*_2$}\\
     $\mathcal{T}$ & The operator for translation $\mathcal{T}: u^l \rightarrow \hat{u}^{*,l}$ \\
     $S_{\eta_1}(u^h(t), \tilde{u}_{128}^h, t)$ & Conditional diffusion model used to approximate $p(u^h|\tilde{u}_{128}^h)$  \\
     $S_{\eta_2}(\tilde{u}_{128}^h(t), \tilde{u}_{64}^h, t)$ & Conditional diffusion model used to approximate $p(\tilde{u}_{128}^h|\tilde{u}_{64}^h)$  \\
     $S_{\eta_3}(\tilde{u}_{64}^h(t), \tu, t)$ & Conditional score model used to approximate $p(\tilde{u}_{64}^h|\tu)$  \\
     $\mS$ & The super-resolution operator: $\mS: \hat{u}^{*,l} \rightarrow \hat{u}^h$ \\
     $\tilde{\mG}^h_{\phi}$ & Neural operator for dynamics at HFLR level \\
     $\mG^h_{\psi}$ & Neural operator for dynamics at HFHR level \\
     $t_1$ & Forward perturbing time \\
     $t_2$ & Backward denoising time \\
     $N$ & Number of LFLR training samples in $\datal$\\
     $M$ & Number of HFHR training samples in $\datah$\\
     $Q$ & \makecell{Number of LFLR testing samples in $\dataltest$ \\ and testing HFHR samples in $\datahtest$} \\
     $n$ & Number of HFHR snapshots in each trajectory\\
     $N'$ & Number of LFLR trajectory training samples in $\edatal$ \\
     $M'$ & Number of HFHR trajectory training samples in $\edatah$ \\
     $Q'$ & \makecell{Number of testing trajectory samples \\ in $\edataltest$ (LFLR) and $\edatahtest$ (HFHR)} \\
     \hline
    \end{tabular}
    \caption{Table of Notations}
    \label{tab:notations}
\end{table}

\begin{table}[H]
    \centering
    \begin{tabular}{c c}
    \hline
      Metric   &   Description \\
    \hline
     TVD & Total Variation Distance \\
     RMSE & Relative root mean squared error \\
     MMD & Maximum Mean Discrepancy\\
     $\mW$ & Wasserstein-2 distance\\
     MELRu & Unweighted mean energy log ratio \\
     MELRw & Weighted mean energy log ratio \\
     \hline
    \end{tabular}
    \caption{Table of Metrics}
    \label{tab:metrics}
\end{table}

\section{Network architecture}\label{sec:archi_hyper}
In this part, we detail the network architectures and the associated hyperparameters used in each method.
\textbf{EDDIB.} We use a UNet architecture for all diffusion models without attention \cite{song2020score} for EDDIB. The noise scheduling function is defined as $\beta(t) = \beta_0 + (\beta_1 - \beta_0) t$ with $\beta_0=0.1, \beta_1=20$ and the dimension of Gaussian random feature embeddings used is set to 128. The remaining hyperparameters for the diffusion models are presented in Table~\ref{tab:hyper}. 
\begin{table}[H]
    \centering
    \begin{tabular}{|c|c|c|c|c|}
        \hline  
        & \makecell{$\mathcal{S}_{\xi}$ and $\tilde{\mathcal{S}}^h_{\zeta}$} & \makecell{$\mathcal{S}_{\eta_1}$} & \makecell{$\mathcal{S}_{\eta_2}$} & \makecell{$\mathcal{S}_{\eta_3}$}  \\
        \hline
        \hline
        Base Channel & 64 & 64 &64 &128 \\
        \hline
        \makecell{Down and Up \\ Channels Multipliers} & 1, 2, 4 & 1, 2, 4, 8 & 1, 2, 4, 8 & 1, 2, 4, 8 \\
        \hline
        Middle Channel & [256, 256] & [512, 512] & [512, 512] & [1024, 1024]  \\
        \hline
        \hline
        Batch Size & 64 & 64 & 64 & 32  \\
        \hline
        Learning Rate & 1e-3 & 5e-4 & 5e-4 & 5e-4 \\
        \hline
        Number of Training Epochs & 2000 & 2000 & 2000 & 2000 \\
        \hline
        \hline
        Number of Denoising Steps & - & 1000 & 1000 & 1000\\
        \hline
    \end{tabular}
    \caption{Hyperparameters of Diffusion Models}
    \label{tab:hyper}
\end{table}

\textbf{SDEdit.} For SDEdit, we require only one single diffusion model, $\tilde{\mathcal{S}}^h_{\zeta}$, trained on the HFLR dataset $\{ \tilde{u}^h_j \}_{j=1}^M$. The hyperparameter settings are identical to those used in EDDIB, with one additional parameter: $t_0$ (see Algorithm 1 in \cite{meng2021sdedit}). This parameter controls the trade-off between realism and faithfulness. In our context, that is to balance the preservation of large-scale structure against the recovery of desired fine-scale details. Here we use $t_0=0.5$.  

\textbf{NOT.} We adopt the ResNet architecture in \cite{liu2019wasserstein} for the potential $f$, which consists of 5 convolutional layers (first three with $5\times5$ kernels, followed by one $3\times3$ and one $1\times1$) with average pooling and ReLU activations. For the stochastic transport map $T(x,z)$ we use an FNO instead of the UNet since FNO better preserves large-scale structure. In our setup, the FNO has 6 layers, 16 modes, and 16 hidden dimensions. We set the learning rate to 1e-4, use a batch size of 64, perform 10 inner iterations ($k_T=10$) per outer iteration, run for a total of 2000 iterations, and use a regularization weight of $\gamma=0.1$ (see Algorithm 1 in \cite{korotin2022neural}).

\textbf{FNO.} For dynamics learning, we use the vanilla FNO \cite{li2020fourier} and the hyperparameters are provided in Table~\ref{tab:hyper_FNO}.
\begin{table}[H]
    \centering
    \begin{tabular}{|c|c|c|}
        \hline  
        &\makecell{$\tilde{\mG}_{\phi}^h$} & \makecell{$\mG_{\psi}^h$} \\
        \hline
        \hline
        Number of Channel (number of snapshots) & 5 & 5 \\
        \hline
        Number of Layers & 4 & 4 \\
        \hline
        Number of Modes & 12 & 32 \\
        \hline
        Hidden Dimensions & 64 & 64 \\
        \hline
        \hline
        Batch Size & 64 & 64  \\
        \hline
        Learning Rate & 1e-4 & 1e-4  \\
        \hline
        Number of Epochs & 500 & 500  \\
        \hline
    \end{tabular}
    \caption{Table of hyperparameter of FNO}
    \label{tab:hyper_FNO}
\end{table}

\section{Additional Algorithms}
In this section, we present two additional algorithms. \Cref{alg:uh} outlines the training procedure of the conditional diffusion models used in the SR step, while \cref{alg:fno} describes the training process of the FNO employed to model the dynamics of trajectory data.

\begin{algorithm}[H]
\caption{Conditional Diffusion Model for SR}
\label{alg:uh}
\begin{algorithmic}[1]
\REQUIRE Paired dataset $\mathcal{T} = \{ \tilde{u}_i^h, u^h_i \}_{i=1}^M$, noise scheduling function $\alpha(t), \sigma(t)$, batch size $B$ and max iteration $Iter$.
\STATE Initialize $k=0$.
\WHILE{$k<Iter$}
    \STATE Sample $\{ \tilde{u}_j^h, u_j^h \}_{j=1}^B \sim \mathcal{T}$ 
    \STATE $t \sim U[0, 1]$
    \STATE $\beps_j \sim \mN(0, \mI)$ for $j=1, \cdots, B$
    \STATE Compute $u_j^h(t) = \alpha(t) u_j^h + \sigma(t) \beps_j$
    \STATE Update $\eta$ using the Adam optimization algorithm to minimize the empirical loss:
    \[
    L(\eta)= \frac{1}{B} \sum_{j=1}^B \|  \beps_j + \sigma(t) S_{\eta}(u_j^h(t),  \tilde{u}_j^h, t) \|_2^2
    \]
    \STATE $k \gets k+1$
\ENDWHILE
\RETURN Diffusion model $S_{\eta}(u^h(t), \tilde{u}^h, t)$
\end{algorithmic}
\end{algorithm}

\begin{algorithm}[H]
\caption{FNO for dynamics}
\label{alg:fno}
\begin{algorithmic}[1]
\REQUIRE Paired dataset $\mathcal{T} = \{ u_{i, 0}, \bu_i \}_{i=1}^N$, batch size $B$ and max iteration $Iter$.
\STATE Initialize $k=0$.
\WHILE{$k<Iter$}
    \STATE Sample $\{ u_{j, 0}, \bu_j \}_{j=1}^B \sim \mathcal{T}$ 
    \STATE Update $\eta$ using the Adam optimization algorithm to minimize the empirical loss:
    \[
    L(\phi)= \frac{1}{B} \sum_{j=1}^B \| \mG_{\phi}(u_{j,0}) - \bu_j \|_2^2
    \]
    \STATE $k \gets k+1$
\ENDWHILE
\RETURN FNO $\mG_{\phi}$
\end{algorithmic}
\end{algorithm}

\section{Metrics}\label{sec:metrics}
In this part, we introduce the metrics used in this paper to quantify the distance between two empirical distributions. These metrics will be used for evaluating the quality of HFHR prediction dataset $\datahath$ obtained from our proposed method and the reference HFHR dataset $\{ u^h_i \}_{i=1}^Q$. First, let us consider the energy spectrum, which measures the energy of a function $f(\bx)$ in each Fourier mode and is defined as 
\[
E_f(k)  = \sum_{| \boldsymbol{k} | = k} \left| \sum_{n} f(\bx_n) \exp \left( -\mathbf{j} 2 \pi \boldsymbol{k} \cdot \bx_n / L \right) \right|^2
\]
where the $\bx_n$ denotes the grid points, the bold notation $\mathbf{j}$ denotes the imaginary number (not to be confused with the index subscript), and $k$ is the magnitude of the wave number $\boldsymbol{k}$. To quantify the overall alignment of two empirical distributions \( \{ \bu_p \}_{p=1}^P \) and \( \{ \bv_q \}_{q=1}^Q \), we denote the average energy spectrum of the two distributions by  
\[
E_{\bu}(k) = \frac{1}{N} \sum_{p} E_{\bu_p}(k), \quad E_{\bv}(k) = \frac{1}{M} \sum_{q} E_{\bv_q}(k),
\]
and consider the mean energy log ratio (MELR):  
\[
\text{MELR}(\bu, \bv)= \sum_k w_k \left| \log \left( \frac{E_{\bu}(k)}{E_{\bv}(k)} \right) \right|
\]  
where \( w_k \) is a user-specific weight function. If \( w_k=1 \) for all \( k \), we refer to it as the unweighted mean energy log ratio (MELRu). If  
\[
w_k=\frac{E_{\bv}(k)}{\sum_{q} E_{\bv}(q)},
\]  
we refer to it as the weighted mean energy log ratio (MELRw). And the definition of the Maximum Mean Discrepancy (MMD), Relative Mean Square Error (RMSE), Wasserstein-2 distance ($\mW$) and Total Variation Distance (TVD) are provided as follow.
\begin{itemize}
     \item[$\circ$] Total Variation Distance (TVD) (two distributions have same size $N=M$): 
    \begin{equation}\label{eqn:tvd}
        \text{TVD}(\bu,\bv) = \frac{1}{N} \sum_{i=1}^N \frac{\| \bu_i - \bv_i \|_1}{\| \bv_i \|_1}
    \end{equation}
	 \item[$\circ$] Relative root mean squared error (RMSE) (two distributions have same size $N=M$):
        \begin{equation}\label{eqn:RMSE}
            \text{RMSE}(\bu, \bv) = \frac{1}{N} \sum_{i=1}^N \frac{\| \bu_i - \bv_i \|_2}{\| \bv_i \|_2}
        \end{equation}
    \item[$\circ$] Maximum Mean Discrepancy (MMD):
    \begin{align}\label{eqn:mmd}
    \text{MMD}(\bu,\bv) =\  
    \frac{1}{N^2} \sum_{i=1}^N \sum_{j=1}^N k(\bu_i, \bu_j) 
    & + \frac{1}{M^2} \sum_{i=1}^M \sum_{j=1}^M k(\bv_i, \bv_j) \nonumber \\
    & - \frac{2}{NM} \sum_{i=1}^N \sum_{j=1}^M k(\bu_i, \bv_j),
    \end{align}
    where
    \[
    k(\bu, \bv) = \exp\left(-\frac{\|\bu - \bv\|^2}{2l^2}\right),
    \]
    is a Gaussian Kernel, and we let $l=0.01$ in this paper.
    
    \item[$\circ$] Wasserstein-2 distance (W2):
    \begin{equation}\label{eqn:w2}
    \text{W2}(\bu,\bv) = \sqrt{\min_{\pi}\sum_{i=1}^N \sum_{j=1}^M \pi_{ij} \| \bu_i - \bv_j \|_2^2},
    \end{equation}
    where $\pi_{ij}$ is the optimal transport plan that satisfies the constraints:
    \[
    \sum_{j=1}^M \pi_{ij} = \frac{1}{N}, ~ \forall i; \quad \text{and} \quad \sum_{i=1}^N \pi_{ij} = \frac{1}{M}, ~ \forall j.
    \]
    In practice, we use the POT package \cite{flamary2021pot} for computation.
   
\end{itemize}

\section{Proofs}\label{sec:proof}
In this section, we provide proofs of Proposition~\ref{prop:1} and Proposition~\ref{prop:2}. Let us start with two lemmas.

\begin{lemma}\label{lemma:1}
Let $p(x, t)$ and $q(x,t)$ be two probability density functions on $(x, t) \in \Omega \times [0, T]$, where $\Omega \subseteq \mathbb{R}^d$ is an open set. Assume $p(x, t)$ and $q(x,t)$ are smooth and fast decaying, that is
\begin{itemize}
    \item $p(\cdot, t),q(\cdot, t) \in C^1(\Omega)$ for all $t \in [0, 1]$,
    \item $\lim_{\|x\| \rightarrow \infty} p(x, t) = 0$, and $\lim_{\|x\| \rightarrow \infty} q(x, t) = 0$,
\end{itemize}
and they both satisfy the same continuity equation:
\begin{equation}
    \begin{cases}
        \frac{\partial p(x,t)}{\partial t} + \nabla \cdot (p(x,t)v(x,t)) = 0, \\
        p(x, 0) = p_0(x)
    \end{cases}
\end{equation}
and 
\begin{equation}
    \begin{cases}
        \frac{\partial q(x,t)}{\partial t} + \nabla \cdot (q(x,t)v(x,t)) = 0, \\
        q(x, 0) = q_0(x)
    \end{cases}
\end{equation}
where $v(x, t)$ is the velocity field. Then we have
\[
\frac{\rd }{\rd t} D_{KL}(p(x, t) \| q(x, t)) = 0.
\]
\end{lemma}
\begin{proof}

\begin{align*}
\frac{\rd }{\rd t} D_{KL}&(p(x, t) \| q(x, t)) \\
&= \frac{\rd }{\rd t} \int p(x, t) \ln \frac{p(x, t)}{q(x, t)} \rd x \\
&= \int \partialt p(x, t) \ln \frac{p(x, t)}{q(x, t)} \rd x 
   + \int \partialt p(x, t) \rd x 
   - \int \partialt q(x, t) \frac{p(x, t)}{q(x, t)} \rd x \\
&= -\int \nabla \cdot (p v) \ln \frac{p}{q} \rd x 
   + 0 
   + \int \nabla \cdot (q v) \frac{p}{q} \rd x \\
&= \int p v \cdot \nabla \left( \ln \frac{p}{q} \right) \rd x 
   - \int q v \cdot \nabla \left( \frac{p}{q} \right) \rd x \\
&= \int p v \cdot \nabla \left( \ln \frac{p}{q} \right) \rd x 
   - \int p v \cdot \nabla \left( \ln \frac{p}{q} \right) \rd x \\
&= 0.
\end{align*}
\end{proof}

\begin{lemma}\label{lemma:2}
Let $p(x, t)$ and $q(x,t)$ be two probability density functions on $(x, t) \in \Omega \times [0, T]$, where $\Omega \subseteq \mathbb{R}^d$ is an open set. Assume $p(x, t)$ and $q(x,t)$ are smooth and fast decaying, that is
\begin{itemize}
    \item $p(\cdot, t),q(\cdot, t) \in C^1(\Omega)$ for all $t \in [0, 1]$,
    \item $\lim_{\|x\| \rightarrow \infty} p(x, t) = 0$, and $\lim_{\|x\| \rightarrow \infty} q(x, t) = 0$,
\end{itemize}
and they both satisfy the same continuity equation:
\begin{equation}
    \begin{cases}
        \frac{\partial p(x,t)}{\partial t} + \nabla \cdot (p(x,t)v(x,t)) = 0, \\
        p(x, 0) = p_0(x)
    \end{cases}
\end{equation}
and 
\begin{equation}
    \begin{cases}
        \frac{\partial q(x,t)}{\partial t} + \nabla \cdot (q(x,t)v(x,t)) = 0, \\
        q(x, 0) = q_0(x)
    \end{cases}
\end{equation}
where $v(x, t)$ is the velocity field, which is $L-$Lipchitz in $x$, we then have 
\[
\mW(p(x, t), q(x, t)) \leq e^{Lt} \mW(p_0(x), q_0(x)).
\]
\end{lemma}
\begin{proof}
    Define a flow map 
    \[
    \Phi_t(x) = X(x, t),
    \]
    where 
    \begin{equation}
        \begin{cases}
            \difft X(x, t) = v(X(x, t), t), \\
            X(x, 0) = x
        \end{cases}
    \end{equation}
    Since the velocity field $v(x, t)$ is $L$-Lipschitz in $x$, by the Gr\"onwall’s inequality, we know that the flow map $\Phi_t$ is $e^{Lt}$-Lipschitz (see \cite{santambrogio2015optimal}), we then have 
    \[
    \mW(p(x, t), q(x, t)) = \mW({\Phi_t}_{\#}p_0(x), {\Phi_t}_{\#}q_0(x)) \leq e^{Lt} \mW(p_0(x), q_0(x)).
    \]
\end{proof}

\begin{proposition}
For any $t_1, t_2 \in (0, 1]$, the KL divergence between the translated distribution $p(\hat{u}(t_1, t_2))$ and the target distribution $p(\tu)$ equals the KL divergence between two perturbed distribution $p(\uul(t_1))$ and $p(\utu(t_2))$, that is
\[
D_{KL}(p(\hat{u}(t_1, t_2)) \| p(\tu)) = D_{KL}(p(\uul(t_1))\|p(\utu(t_2))).
\]
\end{proposition}
\begin{proof}
    The PF ODE
    \begin{equation}
        \begin{cases}
        \frac{ \rd \bx}{\rd t} = - \frac{1}{2} \beta(t) \bx(t) - \frac{1}{2} \beta(t) \tilde{S}^l_\zeta (\bx(t), t), \\
        \bx(0) = \bx_0 \sim p_0(\bx).
    \end{cases}
    \end{equation}
    corresponds to the following continuity equation \cite{santambrogio2015optimal}: 
    \begin{equation}\label{eqn:cont}
        \begin{cases}
        \frac{\partial p(\bx, t)}{\partial t} + \nabla \cdot (p(\bx, t) v(\bx, t)) = 0, \\
        p(\bx, 0) = p_0(\bx).
    \end{cases}
    \end{equation}
    where the velocity field is $v(\bx, t) = \frac{1}{2} \beta(t) \bx(t) - \frac{1}{2} \beta(t) \tilde{S}^l_\zeta (\bx(t), t)$. The distribution $p(\uul(t_1))$ can be obtained by running the continuity equation \eqref{eqn:cont} with initial condition $p_0(\bx) = p(\hat{u}(t_1, t_2))$ for $t_2$. Similarly the distribution $p(\utu(t_2))$ can be obtained by running the continuity equation \eqref{eqn:cont} with initial condition $p_0(\bx) = p(\tu)$ for $t_2$. By applying the Lemma~\ref{lemma:1}, we have
    \[
    D_{KL}(p(\hat{u}(t_1, t_2)) \| p(\tu)) = D_{KL}(p(\uul(t_1))\|p(\utu(t_2))).
    \]
\end{proof}

\begin{proposition}
Assume $\mTt_\zeta(\tu(t), t)$ is $L_s$-Lipchitz continuous in $\tu(t)$. Then,  for any $t_1, t_2 \in (0, 1]$, the $\mW$ distance between the translated distribution $p(\hat{u}(t_1, t_2))$ and the target distribution $p(\tu)$ is upper bounded by the $\mW$ distance between two perturbed distribution $p(\uul(t_1))$ and $p(\utu(t_2))$
\[
\mW(p(\hat{u}(t_1, t_2)), p(\tu)) \leq e^{L_s  t_2} \mW(p(\uul(t_1)), p(\utu(t_2))).
\]
\end{proposition}

\begin{proof}
    The PF ODE
    \begin{equation}
        \begin{cases}
        \frac{ \rd \bx}{\rd t} = \mTt_\zeta(\bx(t), t), \\
        \bx(0) = \bx_0 \sim p_0(\bx).
    \end{cases}
    \end{equation}
    corresponds to the following continuity equation (see \cite{santambrogio2015optimal}): 
    \begin{equation}
        \begin{cases}
        \frac{\partial p(\bx, t)}{\partial t} + \nabla \cdot (p(\bx, t) v(\bx, t)) = 0, \\
        p(\bx, 0) = p_0(\bx).
    \end{cases}
    \end{equation} 
    where the velocity field is $v(\bx, t) = \mTt_\zeta(\bx, t)$ is $L_s$-Lipschitz. The distribution $p(\uul(t_1))$ can be obtained by running the continuity equation \eqref{eqn:cont} with initial condition $p_0(\bx) = p(\hat{u}(t_1, t_2))$ for $t_2$. Similarly the distribution $p(\utu(t_2))$ can be obtained by running the continuity equation \eqref{eqn:cont} with initial condition $p_0(\bx) = p(\tu)$ for $t_2$. By applying the Lemma~\ref{lemma:2}, we have
    \[
\mW(p(\hat{u}(t_1, t_2)), p(\tu)) \leq e^{L_s  t_2} \mW(p(\uul(t_1)), p(\utu(t_2))).
\]
\end{proof}

\end{document}

%% file: ex_shared.tex
\usepackage{lipsum}
\usepackage{amsfonts}
\usepackage{graphicx}
\usepackage{epstopdf}
\usepackage{algorithmic}
\ifpdf
  \DeclareGraphicsExtensions{.eps,.pdf,.png,.jpg}
\else
  \DeclareGraphicsExtensions{.eps}
\fi

\newsiamremark{remark}{Remark}
\newsiamremark{hypothesis}{Hypothesis}
\crefname{hypothesis}{Hypothesis}{Hypotheses}
\newsiamthm{claim}{Claim}
\newsiamremark{fact}{Fact}
\crefname{fact}{Fact}{Facts}

\usepackage{graphicx}
\usepackage{xcolor}
\usepackage{amsmath,amsfonts}
\usepackage{algorithm}
\usepackage{algorithmic}
\usepackage{xcolor}
\usepackage{newfloat}
\usepackage{listings}
\usepackage{makecell}

\newcommand{\mL}{\mathcal{L}}

\newcommand{\mG}{\mathcal{G}}
\newcommand{\mN}{\mathcal{N}}
\newcommand{\mW}{\mathcal{W}_2}
\newcommand{\mI}{\text{I}}
\newcommand{\bx}{\boldsymbol{x}}
\newcommand{\by}{\boldsymbol{y}}
\newcommand{\beps}{\boldsymbol{\varepsilon}}
\newcommand{\rd}{\mathrm{d}}

\newcommand{\mTl}{\mathcal{T}^l}
\newcommand{\mTt}{\tilde{\mathcal{T}}^h}
\newcommand{\mS}{\mathcal{S}}

\newcommand{\bu}{\boldsymbol{u}}
\newcommand{\bv}{\boldsymbol{v}}

\newcommand{\tu}{\tilde{u}^h}
\newcommand{\utu}{\underline{\tilde{u}}^h}
\newcommand{\uul}{\underline{u}^l}

\newcommand{\difft}{\frac{\rd}{\rd t}}
\newcommand{\partialt}{\frac{\partial}{\partial t}}

\newcommand{\datal}{\{ u^l_i \}_{i=1}^N}
\newcommand{\dataltest}{\{ u^l_i \}_{i=1}^Q}
\newcommand{\datah}{\{ u^h_j \}_{j=1}^M}
\newcommand{\datahtest}{\{ u^h_i \}_{j=1}^Q}
\newcommand{\datatest}{\{ u^l_q, u^h_q \}_{q=1}^Q}

\newcommand{\datath}{\{ \tilde{u}^h_j \}_{j=1}^M}
\newcommand{\datathl}{\{ \tilde{u}^h_{j,128} \}_{j=1}^M}
\newcommand{\datathll}{\{ \tilde{u}^h_{j,64} \}_{j=1}^M}

\newcommand{\pdatath}{\{ \tilde{u}^h_{j,64}, \tilde{u}^h_j \}_{j=1}^M}
\newcommand{\pdatathl}{\{ u^h_j, \tilde{u}^h_{j,128} \}_{j=1}^M}
\newcommand{\pdatathll}{\{ \tilde{u}^h_{j,128}, \tilde{u}^h_{j,64} \}_{j=1}^M}

\newcommand{\datahatl}{\{ \hat{u}^{*,l}_i \}_{i=1}^Q}
\newcommand{\datahath}{\{ \hat{u}^h_i \}_{i=1}^Q}

\newcommand{\edatal}{\{ \bu^l_i \}_{i=1}^{N'}}
\newcommand{\edatah}{\{ \bu^h_j \}_{j=1}^{M'}}
\newcommand{\edataltest}{\{ \bu^l_i \}_{i=1}^{Q'}}
\newcommand{\edatahtest}{\{ \bu^h_i \}_{i=1}^{Q'}}
\newcommand{\edatath}{\{ \tilde{\bu}^h_j \}_{j=1}^{M'}}

\headers{Diffusion-based Models for Unpaired SR}{W. Xu, Y. Lu, L. Shen, A. Xuan and A. Barzegari}

\title{Diffusion-Based Models for Unpaired Super-Resoltuon in Fluid Dynamics\thanks{Submitted to the editors DATE.
\funding{YL thanks the support from the National Science Foundation through the award
DMS-2436333 and the support from the Data Science Initiative at the University of
Minnesota through a MnDRIVE DSI Seed Grant.  LS thanks the support from the Office of Naval Research and the University of Minnesota.}}}

\author{Wuzhe Xu\thanks{Department of Mathematics and Statistics, University of Massachusetts Amherst, Amherst, MA 01003, USA (\email{wuzhexu@umass.edu}).}
\and Yulong Lu\thanks{School of Mathematics, University of Minnesota, Minneapolis, MN 55455, USA
  (\email{yulonglu@umn.edu}).} \and Lian Shen\thanks{Department of Mechanical Engineering and Saint Anthony Falls Laboratory, University of Minnesota, Minneapolis, MN 55455, USA (\email{shen@umn.edu}).} \and Anqing Xuan\thanks{Department of Mechanical Engineering and Saint Anthony Falls Laboratory, University of Minnesota, Minneapolis, MN 55455, USA (\email{xuanx004@umn.edu})}. \and Ali Barzegari\thanks{Department of Mechanical Engineering and Saint Anthony Falls Laboratory, University of Minnesota, Minneapolis, MN 55455, USA (\email{barze011@umn.edu}).}}

\usepackage{amsopn}

\makeatletter
\newcommand*{\addFileDependency}[1]{
  \typeout{(#1)}
  \@addtofilelist{#1}
  \IfFileExists{#1}{}{\typeout{No file #1.}}
}
\makeatother
